\documentclass[11pt]{amsart}

\usepackage[T1]{fontenc}
\usepackage{enumerate}
\usepackage{amsaddr}
\usepackage{amsmath}%
\usepackage{amsthm}
\usepackage{amsxtra}%
\usepackage{amsfonts}%
\usepackage{amssymb}%
\usepackage[margin=1in]{geometry}
\usepackage{color,hyperref}
\usepackage{graphicx}
\usepackage{tikz}
\usepackage{subfig}
\usepackage{cite}
\hypersetup{colorlinks,breaklinks,
             linkcolor=blue,urlcolor=blue,
             anchorcolor=blue,citecolor=blue}
\usepackage{color}
\usepackage{array}
\usepackage{booktabs}
\usepackage{dsfont}

\newcommand{\Cr}{\lambda_r}
\newcommand{\cg}{\theta_g}
\newcommand{\cq}{\vartheta_q}

\newcommand{\Lip}{\text{Lip}}
\newcommand{\one}{\mathds{1}}
\renewcommand{\L}{\mathcal{L}}
\renewcommand{\H}{\mathcal{H}}
\newcommand{\usc}{\text{USC}}
\newcommand{\lsc}{\text{LSC}}

\newcommand{\B}{{\mathcal B}}

\renewcommand{\O}{{\mathcal O}}
\renewcommand{\bar}[1]{{\overline{#1}}}
\newcommand{\F}{{\mathcal F}}

\newcommand{\R}{\mathbb{R}}
\newcommand{\N}{\mathbb{N}}

\newcommand{\eps}{\varepsilon}

\renewcommand{\tilde}[1]{\widetilde{#1}}
\renewcommand{\phi}{\varphi}

\renewcommand{\S}{\mathcal{S}}

\def\XXint#1#2#3{{\setbox0=\hbox{$#1{#2#3}{\int}$ }
\vcenter{\hbox{$#2#3$ }}\kern-.6\wd0}}

\newtheorem{theorem}{Theorem}
\newtheorem{lemma}[theorem]{Lemma}

\newtheorem{proposition}[theorem]{Proposition}
\theoremstyle{definition}

\newtheorem{remarkx}[theorem]{Remark}
\newenvironment{remark}{\pushQED{\qed}\remarkx}{\popQED\endremarkx} 
\newtheorem{definition}[theorem]{Definition}

\newcommand{\red}[1]{{\color{red}#1}}

\numberwithin{equation}{section}
\numberwithin{theorem}{section}

\graphicspath{{./images/}}

\begin{document} 

\title[Prediction with history-dependent experts] {Online Prediction with history-dependent experts: The general case}
\author{Nadejda Drenska \and Jeff Calder}
\thanks{{\bf Funding:} Jeff Calder was supported of NSF-DMS grant 1713691.}
\address{School of Mathematics, University of Minnesota}
\email{ndrenska@umn.edu,jcalder@umn.edu}

\maketitle

\begin{abstract}
We study the problem of \emph{prediction of binary sequences} with expert advice in the online setting, which is a classic example of online machine learning. We interpret the binary sequence as the price history of a stock, and view the predictor as an investor, which converts the problem into a \emph{stock prediction problem}. In this framework, an investor, who predicts the daily movements of a stock, and an adversarial market, who controls the stock, play against each other over $N$ turns. The investor combines the predictions of $n\geq 2$ experts in order to make a decision about how much to invest at each turn, and aims to minimize their regret with respect to the best-performing expert at the end of the game. We consider the problem with \emph{history-dependent} experts, in which each expert uses the previous $d$ days of history of the market in making their predictions. We prove that the value function for this game, rescaled appropriately, converges as $N\to \infty$ at a rate of $O(N^{-1/6})$ to the viscosity solution of a nonlinear degenerate elliptic PDE, which can be understood as the Hamilton-Jacobi-Issacs equation for the two-person game. As a result, we are able to deduce asymptotically optimal strategies for the investor.  Our results extend those established by the first author and R.V.~Kohn \cite{drenska2019PDE} for  $n=2$ experts and $d\leq 4$ days of history.
\end{abstract}

\section{Introduction}

Prediction with expert advice refers to a subfield of online machine learning  \cite{CBL}. It models real world situations where an investor uses \emph{expert advice} to predict against (or play against) an adversarial market. In particular, there is a multistep process where new information becomes available at every time step and a learner (or investor) tries to incorporate this data into sequential decisions.  Pioneering works in the machine learning literature for prediction with expert advice are Cover's \cite{cover1966behavior} and Hannan's \cite{Hannan} papers. Various heuristic approaches that achieve good results are contained in \cite{ABG, CBL, GPS, HKW, LW, CFH, Ro}, and recent work has focused on provably optimal strategies \cite{drenska2017pde,drenska2020prediction,drenska2019PDE, ABG, GPS, Bayraktar}.  Typical applications of prediction with expert advice include stock price prediction, portfolio optimization \cite{FS}, self-driving car software \cite{AKT}, and algorithm boosting \cite{FS}.

We consider the problem of \emph{prediction of binary sequences} with expert advice in the online setting. As in \cite{drenska2019PDE}, we call the problem a \emph{stock prediction problem}, since we think of the predictor as an investor, and the binary sequence as the price history of a stock. We measure how effective the investor's strategy is through the notion of \emph{regret}, which is the difference between the investor's performance and the performance of an expert. \emph{Prediction} refers to the process by which the investor combines the advice of multiple experts to make their own investment decision. The investor's goal is to minimizing regret with respect to the best performing expert, and thus obtain provably good performance. An underlying assumption is that each expert may have a varying degree of predictive ability. Indeed, some experts may be poor predictors, some may be adversarial, and some may have inside information and perform above average often.  The central question becomes how to distinguish between the different experts and take advantage of the best performing ones. In this paper, we take the commonly used assumption that the market is \emph{adversarial}, and is thus another player in the game whose goal is to maximize the investor's regret. In other words, we are undertaking a worst case analysis.

We are interested in the case of \emph{history-dependent experts}, in which each expert uses the previous $d$ days of market history to make their predictions.  The case with two static experts---one optimistic (who always bids $+1$) and one pessimistic (who always bids $-1$)---was first introduced by Thomas Cover in 1966  \cite{cover1966behavior}.  Recent work  has considered PDE scaling limits in the static case \cite{KP, AP,  Zhu}, and the first author and R.V.~Kohn \cite{drenska2019PDE} recently extended these results by allowing the two experts' behaviors to be history-dependent. This extension introduces a second time scale, so the system becomes `fast-slow', with a `fast' variable living on a discrete graph that describes the market history. In order to handle this complication, \cite{drenska2019PDE} used ideas from graph theory and was able to completely solve the problem for $n=2$ experts and $d\leq 4$ days of market history, and establish upper and lower bounds for the value function for $n=2$ and $d\geq 5$.

In this paper, we extend the results of \cite{drenska2019PDE} to any number of experts $n\geq 2$ and any number of days $d\geq 1$ of market history. In particular, we prove that the value function for the discrete prediction problem converges, with quantitative rates, to the viscosity solution of a nonlinear degenerate elliptic PDE. The PDE is the same as the one in \cite{drenska2019PDE} for $n=2$ experts. We then use the solution of the PDE to construct a provably asymptotically optimal strategy for the investor. A key feature of our work is that the prediction problem is played over a graph, which encodes the ways in which the $d$ days of market history can transition at each step of the game. The graph is the $d$-dimensional de Bruijn graph over $2$ symbols (see Figure \ref{fig:graph}). The value function for the two-person game varies rapidly over the graph, introducing a `fast' variable, and in order to understand the long-time behavior of the game, we have to understand how the fast variable averages out in the long run. Our proof utilizes a $k$-step dynamic programming principle, instead of the usual $1$-step version. For $k$ sufficiently large, the `fast' variable averages out over the graph. It is possible to view our proof through the lens of homogenization theory. Indeed, the \emph{local problem} we identify in Section \ref{sec:cell} is essentially a cell problem, and describes the local oscillations of the value function. Our approach is completely different from the one used in \cite{drenska2019PDE}, which works with two linear programs related to movement on the de Bruijn graph. In particular, the convergence rates that we obtain are worse by a cube root from those established in \cite{drenska2019PDE} for $n=2$ and $d\leq 4$. We refer to Section \ref{sec:overview} for a more thorough comparison of our work with \cite{drenska2019PDE}.

There are many other cases in the PDE literature where scaling limits of sequential decision making result in elliptic or parabolic PDEs. Examples include the Kohn-Serfaty two-person game for curvature motion \cite{KS1}, which can be extended to more general equations \cite{KS2}, and the stochastic tug-of-war games for the $p$-Laplacian and $\infty$-Laplacian \cite{PS2,PS1}. These works have been followed by many others (see e.g. \cite{AS1, NS, APSS, LM,calder2020convex}). In particular, our work is somewhat related to \cite{calder2020convex}, in which the second author and C.K.~Smart prove that convex hull peeling has a continuum limit that corresponds to affine invariant curvature motion. Convex hull peeling has an interpretation as a two-person game played on a random point cloud. In \cite{calder2020convex}, the authors also use a multistep approach, where a large number of steps in the dynamic programming principle are required to ensure the value function averages out locally.

This paper is organized as follows. In Section \ref{sec:setup} we describe the setup for prediction with history dependent experts, and in Section \ref{sec:main} we state our main results.  In Section \ref{sec:overview} we give an overview of the main ideas behind our proofs, and how they relate to the previous work by the first author and R.V.~Kohn \cite{drenska2019PDE}. In Section \ref{sec:game} we study the discrete value function and establish basic properties, including the $k$-step dynamic programming principle. In Section \ref{sec:cell}, we study what we call the \emph{local problem}, which arises from Taylor expansion in the $k$-step dynamic programming principle, and show that the local problem converges as $k\to \infty$ at the rate $O\left( \frac{1}{k} \right)$. In Section \ref{sec:pde}, we study the continuum PDE, proving existence of a unique linear growth viscosity solution under mild assumptions, and establishing regularity in some special cases. Finally, in Section \ref{sec:proofs} we give the proofs of our main results.

\subsection{Setup}
\label{sec:setup}

We follow the setup in \cite{drenska2019PDE}. Assume we have $n\geq 2$ experts making predictions about the movement of a particular stock. The change in stock price on a daily basis is described by a stream of binary data $b_1,b_2,b_3,\dots,b_i,\dots$ with $b_i\in \B:=\{-1,1\}$, representing whether the stock increased or decreased on day $i$. Every day, each of the $n$ experts makes a prediction about whether the stock will increase or decrease tomorrow. The investor uses these predictions to make an investment, and this yields a corresponding gain or loss, depending on the movement of the market $b_i$. The game is played for a fixed number of days $N$, and the performance of the player is compared against the best performing expert.

We assume the $n$ experts each use a fixed publicly available algorithm to make their predictions, and the predictions depend on the previous $d$ days of history of stock movement. That is, on day $i$, the experts use the data 
\begin{equation}\label{eq:mk}
m^i:=(b_{i-d},b_{i-d+1},\dots,b_{i-1}) \in \B^d
\end{equation}
to make a prediction about $b_i$. The $n$ expert predictions are taken to be fixed functions
\begin{equation}\label{eq:experts}
q_1,\dots,q_n:\B^d\to [-1,1],
\end{equation}
where $q_j(m)$ represents the prediction of expert $j$ given stock history $m\in \B^d$. The predictions are real numbers in the interval $[-1,1]$, indicating the confidence each expert has in their prediction. For notational convenience we write $q:=(q_1,\dots,q_n):\B^d \to [-1,1]^n$ for the vector of all expert predictions. We assume the predictions $q(m)$ are publicly known for all $m\in \B^d$.  Given the expert predictions $q(m^i)$ of $b_i$, the investor decides on an investment $f_i\in [-1,1]$, which can be interpreted as an amount of the stock to buy or sell. The market then chooses $b_i\in \B$. If $b_i=1$, then the investor gains $f_i$, while if $b_i=-1$ then the investor loses $f_i$. Thus, the investor gains $b_if_i$ on day $i$. Similarly, the $j^{\rm th}$ expert, were they to invest their prediction, would gain $b_iq_j(m^i)$. 

The investor's performance is measured by their \emph{regret} against each expert. The regret relative to an expert is the difference between the gains of the expert and that of the investor. We denote by $x_i\in \R$ the regret of the investor with respect to expert $i$, and write $x=(x_1,\dots,x_n)\in \R^n$ for the vector of regrets with respect to all experts. The change in regret with respect to expert $j$ on day $i$ is thus $b_i(q_j(m^i)-f_i)$. In the context of prediction, one would say we are using the financial loss function
\[L(f_i,b_i) := b_i(q_j(m^i) - f_i).\]
For more general prediction problems, other losses for measuring how well the investor predicts $b_i$ could be used (e.g., $L(f_i,b_i)=|f_i-b_i|$). We expect the results and techniques used in this paper to apply to other losses as well, with some modifications. It is also important to point out that we do not index the regret by the day $i$. In this framework, the regret is a \emph{state variable}, and the change in regret is realized as moving the game to a new state. 

After the game is played for $N$ days, the investor's regret is evaluated with a \emph{payoff function} $g:\R^n\to \R$. A common choice is $g(x)=\max\{x_1,\dots,x_n\}$, which simply reports the regret compared to the best performing expert. While the maximum regret is most commonly used in practice, our anlaysis works for more general payoffs, satisfying reasonable conditions, so we proceed in generality. The goal of the investor is to minimize $g(x)$, where $x$ is the regret vector at the end of the game. The market is assumed to be adversarial, and is selecting the stock movements $b_i$ so as to maximize $g(x)$.  Thus, we are undertaking a \emph{worst case analysis} in this paper.

Underlying the two-player game is a directed graph that encodes the ways in which the history $m^i$ can change from day to day. At each step $i$ of the game, there are only two possible states for the history window at step $i+1$, depending on whether $b_i=1$ or $b_i=-1$. In order to describe this graph we introduce some notation. For $m=(m_1,\dots,m_d)\in \B^d$ and $b\in \B$ we define $m|b\in \B^d$ by
\begin{equation}\label{eq:concat1}
m|b:=(m_2,m_3,\dots,m_d,b).
\end{equation}
In this notation, the history window $m^i$ evolves according to $m^{i+1}=m^i|b_i$. We also write $m_+=m|1$ and $m_-=m|-1$. Each node in the graph is a possible state $m\in \B^d$ of the game's history, and there is a directed edge from $m$ to $m_+$ and from $m$ to $m_-$ for every node $m$.  This graph is called the $d$-dimensional de Bruijn graph over $2$ symbols. Figure \ref{fig:graph} shows the $3$-dimensional de Bruijn graph, where we have written $0$ in place of $-1$ to simplify the figure.  The presence of this underlying de Bruijn graph creates additional challenges in describing the optimal strategies and optimal value for the game.
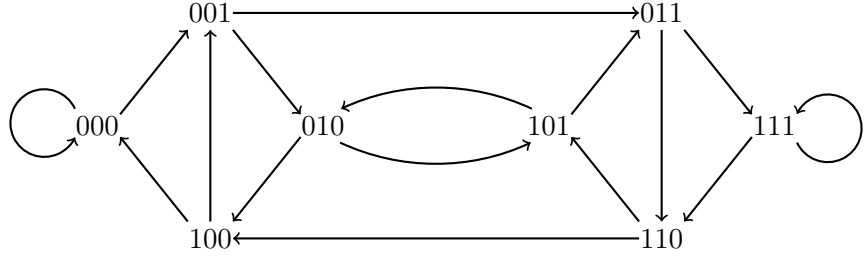
\begin{figure}
\begin{center}
\begin{tikzpicture}[scale=1.5]
\node at (-1,0) {010};
\node at (1,0) {101};
\node at (-3,0) {000};
\node at (3,0) {111};
\node at (-2,-1) {100};
\node at (-2,1) {001};
\node at (2,-1) {110};
\node at (2,1) {011};
\draw[thick,->] (-3.2,0.15) arc (25:335:0.3);
\draw[thick,->] (3.2,-0.15) arc (-180+25:180-25:0.3);
\draw[thick,->] (-2.8,0.1)--(-2.2,0.85);
\draw[thick,<-] (-2,0.85)--(-2,-0.85);
\draw[thick,->] (-1.8,0.85)--(-1.2,0.1);
\draw[thick,<-] (-1.8,-0.85)--(-1.2,-0.1);
\draw[thick,<-] (-2.8,-0.1)--(-2.2,-0.85);
\draw[thick,<-] (2.8,0.1)--(2.2,0.85);
\draw[thick,->] (2.8,-0.1)--(2.2,-0.85);
\draw[thick,->] (2,0.85)--(2,-0.85);
\draw[thick,<-] (1.8,0.85)--(1.2,0.1);
\draw[thick,->] (1.8,-0.85)--(1.2,-0.1);
\draw[thick,->] (-1.8,1)--(1.8,1);
\draw[thick,<-] (-1.8,-1)--(1.8,-1);
\draw[thick,->] (0.85,0.15) arc (90-25:90+25:2);
\draw[thick,->] (-0.85,-0.15) arc (270-25:270+25:2);
\end{tikzpicture}
\end{center}
\caption{The de Bruijn Graph, $d=3$}
\label{fig:graph}
\end{figure}

The discussion above was largely informal. To be precise, we now define the \emph{value function}.
\begin{definition}[Value function]\label{def:value}
Let $g:\R^n\to \R$. Given $N\in \N$, $m\in \B^d$, and $1 \leq \ell \leq N$, the \emph{value function} $V_N(x,\ell;m)$ is defined by $V_N(x,\ell;m)=g(x)$ for $\ell=N$, and  
\begin{equation}\label{eq:valuedef}
V_N(x,\ell;m) = \min_{|f_{\ell}|\leq 1}\max_{b_{\ell}=\pm 1}\min_{|f_{\ell+1}|\leq 1}\max_{b_{\ell+1}=\pm 1}\cdots \min_{|f_{N-1}|\leq 1}\max_{b_{N-1}=\pm 1}g\left( x + \sum_{i=\ell}^{N-1} b_i(q(m^i) - f_i\one) \right)
\end{equation}
for $1 \leq \ell \leq N-1$, where $m^\ell = m$ and $m^{i+1}=m^i|b_i$ for $i=\ell,\dots,N-1$.
\end{definition}
Here, we use the notation $\one$ for the all ones vector $\one=(1,1,\dots,1)\in \R^n$. The value of $V_N(x,\ell;m)$ is the payoff on the final day $N$, given the game starts on day $\ell$ with regret $x\in \R^n$ and history $m\in \B^d$, and both the investor and market play optimally.  Notice there are, in fact, $2^d$ value functions, one for each $m\in \B^d$. 

\subsection{Main results}
\label{sec:main}

We are interested in understanding the long-time behavior of the value functions as $N\to \infty$, and the asymptotically optimal investor strategies. For this, we place the following structural assumptions on the payoff. 
\begin{flalign}\label{eq:strict_inc}\tag{G1} 
\text{There exists } \cg>0 \text{ such that for all } x\in \R^n, v\in [0,\infty)^n,\, g(x+v) \geq g(x) + \cg \langle v,\one\rangle,&&
\end{flalign}
\begin{flalign}\label{eq:homogeneous}\tag{G2} 
\text{For all } x\in \R^n, s>0 \text{ we have } g(sx) = sg(x).&&
\end{flalign}
We also place the following assumption on expert strategies.
\begin{flalign}\label{eq:redund}\tag{E1} 
\text{For all }m\in \B^d,\, q(m)\neq \one \text{ and } q(m)\neq -\one.&&
\end{flalign}
Assumption \eqref{eq:redund} asks that the experts never all agree at $+1$ or $-1$.  For example, if one expert always predicts $+1$ while another always predicts $-1$, then \eqref{eq:redund} holds. This assumption guarantees that the constant $\cq$ defined by
\begin{equation}\label{eq:thetaq_def}
\cq := \min_{m\in \B^d}\min\left\{\sum_{i=1}^n (1-q_i(m)_+),\sum_{i=1}^n (1-q_i(m)_-) \right\}
\end{equation}
is strictly positive $\cq>0$, where $a_+=\max\{a,0\}$ and $a_-=-\min\{a,0\}$. 

To obtain a meaningful continuum limit, we must rescale $V_N$ appropriately. We define the rescaled value function $u_N:\R^n\times [0,1]\times \B^d\to \R$ by
\begin{equation}\label{eq:uN}
u_N(x,t;m) := \frac{1}{\sqrt{N}}V_N(\sqrt{N}x,\lceil Nt\rceil;m).
\end{equation}
Here, $\lceil t\rceil$ denotes the smallest integer greater than $t$. The rescaling in \eqref{eq:uN} is \emph{parabolic} rescaling, and is due to the adversarial nature of the problem, which causes $O(\sqrt{N})$ regret to accumulate after $N$ steps of the game. We also define the upper and lower value functions $u_N^+$ and $u_N^-$ by
\begin{equation}\label{eq:uNupper}
u^{+}_N(x,t) = \max_{m\in \B^d}u_N(x,t;m) \ \ \text{ and } \ \ u^{-}_N(x,t) = \min_{m\in \B^d}u_N(x,t;m).
\end{equation}

Our main results, given below, show that $u_N^\pm$ converge uniformly, with convergence rates, to the solution of the continuum PDE
\begin{equation}\label{eq:PDE}
\left\{\begin{aligned}
u_t + \frac{1}{2^{d+1}}\sum_{\eta\in Q(\nabla u)} \langle \nabla^2 u\,\eta,\eta\rangle &= 0,&&\text{in }\R^n\times (0,1)\\ 
u &=g,&&\text{on }\R^n\times \{t=1\},
\end{aligned}\right.
\end{equation}
where for $p\in \R^n$ the set $Q(p)$ is given by 
\begin{equation}\label{eq:Qdef}
Q(p)=\left\{q(m) - \frac{\langle p, q(m)\rangle}{\langle p, \one\rangle}\one \, : \, m\in \B^d\right\}
\end{equation}
when $\langle p, \one\rangle\neq 0$, and $Q(p)=\varnothing$ otherwise. Essentially, \eqref{eq:PDE} is the limiting Hamilton-Jacobi-Isaacs equation for the two player game.  Since $Q(p)\subset p^\perp$, \eqref{eq:PDE} is a degenerate diffusion equation. 

Our first result is the following continuum limit.
\begin{theorem}\label{thm:main1}
Let $n\geq 2$. Let $g$ be uniformly continuous, and assume \eqref{eq:strict_inc}, \eqref{eq:homogeneous} and \eqref{eq:redund} hold. Let $u\in C(\R^n\times [0,1])$ be the unique viscosity solution of \eqref{eq:PDE}.  As $N\to \infty$ we have
\[u^\pm_N \longrightarrow u \text{ uniformly on }\R^n\times [0,1].\]
Furthermore, if $g\in C^4(\R^n)$ with $[g]_{C^4(\R^n)}<\infty$, then there exists $C_1,C_2>0$ depending on $n$, $\cg$, and $[g]_{C^4(\R^n)}$, such that for all $t\in [0,1]$ and
\begin{equation}\label{eq:Ncond}
N \geq \max\left\{ \frac{(d+1)^6}{d^2},\frac{C_1d}{\cq^{3}} \right\}
\end{equation}
it holds that
\[\sup_{x\in \R^n}|u^\pm_N(x,t) - u(x,t)|\leq C_2 \left((1-t)d^{2/3}N^{-1/6}  + d^{1/3}N^{-1/3}\right).\]
\end{theorem}
We show in Section \ref{sec:pde} that when $g$ is uniformly continuous and \eqref{eq:strict_inc} holds, \eqref{eq:PDE} has a unique linear growth viscosity solution. We also recall the $C^k(\R^n)$ semi-norm of $u$ is defined as
\begin{equation}\label{eq:cksemi}
[g]_{C^k(\R^n)} = \sup_{x\in \R^n}\max_{1\leq |\alpha|=k}|D^\alpha g(x)|.
\end{equation}

It is also common in the literature on online learning to assume the payoff satisfies the following translation property:
\begin{flalign}\label{eq:translation}\tag{G3} 
\text{For all }x\in\R^n \text{ and }s\in \R, \ \  g(x+s\,\one) = g(x) + s.&&
\end{flalign}
When the translation property holds, the rate in Theorem \ref{thm:main1} can be extended to Lipschitz continuous payoffs $g$. This includes the commonly used payoff $g(x) = \max\{x_1,\dots,x_n\}$, which corresponds to measuring regret with respect to the best performing expert. For this, we need to place an additional assumption on the expert strategies.  We define $r:\B^d \to \R^{n-1}$ by
\begin{equation}\label{eq:rdef}
r(m)= (q_1(m)-q_n(m),\dots,q_{n-1}(m)-q_n(m)),
\end{equation}
and we assume
\begin{flalign}\label{eq:span}\tag{E2} 
\text{There exists } 0 < \Cr \leq 1 \text{ such that } \frac{1}{2^{d+1}} \sum_{m\in \B^d} r(m)\otimes r(m) \geq \Cr I, &&
\end{flalign}
where $I$ is the $(n-1)\times (n-1)$ identity matrix. We recall that for symmetric matrices $A$ and $B$, the notation $A\geq B$ means that $A-B$ is positive semi-definite. 

In this case, we have the following result.
\begin{theorem}\label{thm:main2}
Let $n\geq 2$. Let $g$ be Lipschitz continuous, and assume \eqref{eq:strict_inc}, \eqref{eq:homogeneous}, \eqref{eq:translation}, \eqref{eq:redund} and \eqref{eq:span} hold. Let $u\in C(\R^n\times [0,1])$ be the unique viscosity solution of \eqref{eq:PDE}.  Then there exists $C_1,C_2>0$ depending only on $n$, such that for all $t\in [0,1]$ and
\begin{equation}\label{eq:Ncond2}
N \geq \frac{C_1(d+1)^6}{d^2\Cr}
\end{equation}
it holds that
\[\|u^\pm_N - u\|_{L^\infty(\R^n\times [0,1])}\leq C_2\Lip(g)\left( 1 + \frac{\Lip(g)^2}{\cg^2\cq^2}+\log\left(1+d^{-1/3}\Cr^{-1/6}N^{5/6}\right) \right) \Cr^{-2/3}d^{2/3}N^{-1/6}.\]
\end{theorem}
We recall the Lipschitz constant of $g$ is given by
\[\Lip(g) = \sup_{\substack{x,y\in \R^n \\ x\neq y}}\frac{|g(x)-g(y)|}{|x-y|}.\]

Several remarks are in order.
\begin{remark}
In the proofs of Theorems \ref{thm:main1} and \ref{thm:main2}, we end up obtaining asymptotically optimal strategies for the investor and market. We show that the investor's optimal strategy is the one that achieves indifference to the market's choice $b_i$, while the market's optimal strategy is to choose $b_i$ to penalize any deviation from the investor's optimal strategy. The proofs in our paper do not explicitly use these strategies; instead, our proofs are concerned with the optimal \emph{value}, given optimal strategies are employed. For reference, we describe an asymptotically optimal investor strategy below, which is a byproduct of the proof of Lemma \ref{lem:cell} in Section \ref{sec:cell}.

Let the initial regret on day $1$ be denoted $x^1\in \R^n$, and the initial history window be denoted $m^1\in \B^d$. Let 
\[x^j = x^1 + \sum_{i=1}^{j-1} b_i(q(m^i) - f_i\one)\]
be the regret on day $j$, where $m^{i+1}=m^i|b_i$. Let $1 \ll k \ll N$ such that $k$ divides evenly into $N$, and consider dividing the number of plays of the game $N$ into blocks of size $k$. We describe the strategy on the $\ell^{\rm th}$ block $f_{\ell k+1},f_{\ell k+2},\dots,f_{(\ell+1)k}$. We compute the solution of \eqref{eq:PDE} and set
\[p = \nabla u(x^{\ell k}) \ \ \text{ and }\ \ X = \nabla^2 u(x^{\ell k}).\]
  We define $\H_i:\B^d\to \R$ by $\H_0(m)=0$ for all $m\in \B^d$ and the recursion
\[\H_{i}(m) = \frac{1}{2}\langle X\xi(m),\xi(m)\rangle + \frac{1}{2}\left( \H_{i-1}(m_+) + \H_{i-1}(m_-) \right),\]
for $i\geq 1$, where
\[\xi(m) = q(m) - \frac{\langle p, q(m)\rangle}{\langle p, \one\rangle}\one.\]
See Proposition \ref{prop:S} for more properties of $\H_i$. Then for $i=1,\dots,k$ the investor chooses the strategy
\begin{multline}\label{eq:opt_strategy}
f_{\ell k +i} = h_{\ell k +i}^{-1}\Bigg( \langle p,q(m^{\ell k+i})\rangle + \eps\sum_{j=\ell k +1}^{\ell k+i-1}b_j\langle Xq(m^{\ell k+i}),q(m^{j})-\one f_{j}\rangle\\
+ \frac{\eps}{2}(\H_{k-i}(m^{i}_+) - \H_{k-i}(m^{i}_-))\Bigg),
\end{multline}
where
\[h_{\ell k +i}:=\langle p,\one\rangle + \eps\sum_{j=\ell k + 1}^{\ell k+i-1}b_j\langle X\one,q(m^j) - \one f_j\rangle.\]
where we have set $\eps = N^{-1/2}$ for convenience. 
This investor strategy makes the market indifferent (in an asymptotic sense) to $b_i=\pm 1$. The proof of this is contained in Lemma \ref{lem:cell}.  The amount of accumulated regret after following this investor strategy for all $k$ steps of the $\ell^{\rm th}$ block is approximately $\H_k(m^1)$. This turns out to correspond to a weighted average of $\tfrac12 \langle X\xi(m),\xi(m)\rangle$ over a de Bruijn tree of depth $k$ rooted at $m^1$, and as $k\to \infty$ this tree averages out over the de Bruijn graph, yielding (see Proposition \ref{prop:S})
\[\frac{1}{k}\H_k(m^1) \sim \frac{1}{2^{d+1}}\sum_{m\in \B^d} \langle X\xi(m),\xi(m)\rangle.\]
Notice this is the same operator appearing in our main PDE \eqref{eq:PDE}. Any choice of $1 \ll k \ll N^{1/2}$ yields an asymptotically optimal strategy. In the proof of our main results, we optimize over the choice of $k$, yielding $k\sim d^{1/3}N^{1/6}$.

Let us remark that the strategy \eqref{eq:opt_strategy} on the first step of a new block ($i=1$) is given by
\begin{equation}\label{eq:candidate}
f_{\ell k +1} = \frac{\langle p,q(m^{\ell k+1})\rangle}{\langle p,\one\rangle}  + \frac{\eps}{2}\left(\frac{\H_{k-1}(m^{1}_+) - \H_{k-1}(m^{1}_-)}{\langle p,\one\rangle}\right).
\end{equation}
As we show in Proposition \ref{prop:S}, the term
\begin{equation}\label{eq:candidate2}
\H_{k-1}(m^{1}_+) - \H_{k-1}(m^{1}_-)
\end{equation}
is independent of $k$, as long as $k\geq d+1$. This term is exactly the difference of weighted sums over de Bruijn trees of depth $k-1$ rooted at $m^1_+$ and $m^1_-$ (we refer to Proposition \ref{prop:S} for more details).  In a followup paper \cite{calder2020asymp}, we show that a strategy of this form is also asymptotically optimal for the investor, but with shaper $O(\eps)$ convergence rates.


We also mention that, unlike in \cite{drenska2019PDE}, the asymptotically optimal investor strategy we identified in \eqref{eq:opt_strategy} is not given by an explicit formula, since it involves the partial derivatives of the value function $u$, which is characterized as the unique solution of the nonlinear parabolic PDE \eqref{eq:PDE}. When $n=2$ and $g(x)=\max_i x_i$, it was shown in \cite{drenska2019PDE} that this PDE can be solved analytically, giving explicit formulas for the optimal strategies in this case. For $n\geq 3$, even when $g(x)=\max_i x_i$, we are not able to solve the equation in closed form. However, in  Theorem \ref{thm:reg2}, we show that whenever $g$ satisfies the translation property \eqref{eq:translation}, the PDE \eqref{eq:PDE} admits a representation formula for the solution in terms of a convolution and a linear change of coordinates. While this respresentation formula is not explicit, it may be possible to numerically approximate the convolution, even in high dimensions, with Monte-Carlo methods. We leave this to future work. 
\label{rem:optimal}
\end{remark}

\begin{remark}
We briefly remark on the roles of the hypotheses \eqref{eq:strict_inc}, \eqref{eq:homogeneous}, and \eqref{eq:redund}. First, \eqref{eq:homogeneous} is only used to ensure the final time condition $u_N(x,1;m) = g(x)$ holds. If instead of defining $u_N$ as in \eqref{eq:uN}, we use the alternative rescaled definition
\[u_N(x,t;m) = \min_{|f_{\lceil Nt\rceil}|\leq 1}\max_{b_{\lceil Nt\rceil}=\pm 1}\cdots \min_{|f_{N-1}|\leq 1}\max_{b_{N-1}=\pm 1}g\left( x + N^{-1/2}\sum_{i=\ell}^{N-1} b_i(q(m^i) - f_i\one) \right),\]
then we can omit the hypothesis \eqref{eq:homogeneous}.  If \eqref{eq:homogeneous} does not hold, and we define $u_N$ as in \eqref{eq:uN}, then we expect a result similar to Theorem \ref{thm:main1} to hold, provided we replace $g$ in \eqref{eq:PDE} with 
\[g_0(x):=\lim_{\eps\to 0}\eps g\left( \frac{x}{\eps} \right),\]
provided the limit exists. To obtain the same convergence rate as in Theorem \ref{thm:main1}, we would have to assume a rate of convergence as $\eps\to 0$ in the definition of $g_0$ above.

Second, while the conditions \eqref{eq:redund} and \eqref{eq:strict_inc} appear in the convergence rate in Theorem \ref{thm:main1} through the constants $\cq$ and $\cg$, it appears these conditions are necessary even for the convergence $u_N^\pm \to u$ without a rate. To see why, we show in Proposition \ref{prop:PDEprop} (ii) that \eqref{eq:strict_inc} implies that $u_{x_i}\geq \cg >0$ for all $i$. Combining this with  \eqref{eq:redund} we see that
\[-1 < \frac{\langle \nabla u ,q(m)\rangle}{\langle \nabla u,\one\rangle} < 1\]
holds for all $m\in \B^d$.  Thus, when $N$ is sufficiently large, so that $\eps = N^{-1/2}$ is sufficiently small, the optimal investor strategy $f_i$ given in \eqref{eq:opt_strategy} (note $p=\nabla u$) is guaranteed to be admissible; that is, it lies in the interval $f_i\in [-1,1]$. If there are nodes $m\in \B^d$ in the de Bruijn graph where $q(m)=\one$ or $q(m)=-\one$, then the optimal strategy \eqref{eq:opt_strategy} may sometimes be inadmissible for the investor. In this case, the investor will be unable to render the market indifferent to $b_i=1$ or $b_i=-1$, and as a result, the market can exploit the investor and accumulate additional regret. The condition \eqref{eq:redund} is not needed if we allow the investor more flexibility in their investment, and invest $f_i \in [-1-\delta,1+\delta]$ for some $\delta>0$. 

We note that we still expect to see some kind of continuum limit result even when \eqref{eq:redund} does not hold, however, the limiting PDE \eqref{eq:PDE} may have a different form. In particular, instead of an equal weighting over all nodes in the de Bruijn graph, we expect that nodes with $q(m)=\one$ or $q(m)=-\one$ may be more heavily weighted, indicating that these nodes contribute a higher amount of regret. We also mention that \eqref{eq:strict_inc} is used to show that the PDE \eqref{eq:PDE} has a unique viscosity solution, although the weaker condition $\langle \nabla g,\one\rangle \geq \cg>0$ is sufficient for this purpose. 
\label{rem:hypotheses}
\end{remark}

\begin{remark}
Notice in Theorem \ref{thm:main2}, the constants $C_1$ and $C_2$ depend only on the number of experts $n$. In particular, the dependence on the dimension $d$ of the de Bruijn graph is recorded explicitly and is \emph{sublinear} (i.e., $d^{2/3}$) in the convergence rate, while polynomial in the condition \eqref{eq:Ncond2} on $N$. A similar comment is true for Theorem \ref{thm:main1}, though the constants in that theorem depend additionally on regularity properties of $g$.
\label{rem:constants}
\end{remark}

\begin{remark}
\label{rem:classical}
It is not common in the literature on scaling limits for two-player games to obtain convergence rates as in Theorems \ref{thm:main1} and \ref{thm:main2}, due to a lack of regularity for the viscosity solution of the limiting equation \eqref{eq:PDE}. In this case, the PDE \eqref{eq:PDE} has a hidden geometric structure that allows us to prove that the viscosity solution $u$ is classical, in certain cases, with sufficient control on its derivatives to obtain the convergence rates. In particular, the PDE \eqref{eq:PDE} is a geometric equation that describes the evolution of the level sets of $u$ by a heat equation. In the right coordinate system, the heat equation is \emph{linear} and \eqref{eq:span} is exactly the corresponding uniform ellipticity condition. This was first observed for $n=2$ experts in the work of Zhu \cite{Zhu}, and this observation also plays an essential role in \cite{drenska2019PDE}. We refer to Theorems \ref{thm:reg2} and \ref{thm:reg1} for the general statements (for any $n\geq 2$) of this geometric structure.

In fact, when the translation property \eqref{eq:translation} holds, it is straightforward to see where the additional regularity comes from. Indeed, \eqref{eq:translation} implies that $u$ also satisfies the translation property (see Proposition \ref{prop:PDEprop} (iii)) and so, formally speaking, $\langle \nabla u,\one\rangle = 1$. Differentiating again we obtain $\nabla^2 u\one =0$. Therefore, the equation \eqref{eq:PDE} simplifies to the linear heat equation 
\begin{equation}\label{eq:PDElinear}
\left\{\begin{aligned}
u_t + \frac{1}{2^{d+1}}\sum_{m\in \B^d} \langle \nabla^2 u\,q(m),q(m)\rangle &= 0,&&\text{in }\R^n\times (0,1)\\ 
u &=g,&&\text{on }\R^n\times \{t=1\}.
\end{aligned}\right.
\end{equation}
If $\sum_{m\in \B^d}q(m)\otimes q(m) \geq \lambda I$, then \eqref{eq:PDElinear} is uniformly elliptic and $u\in C^\infty(\R^n\times [0,1))$. We note that the uniform ellipticity condition \eqref{eq:span} is for a different equation (see Theorem \ref{thm:reg2} and Remark \ref{rem:otherspan}) that is obtained by using the translation property to reduce the dimension to $n-1$. The condition \eqref{eq:span} is implied by uniform ellipticity of \eqref{eq:PDElinear}, and is hence a weaker condition. We also note that \eqref{eq:span} implies that the vectors $\{r(m)\}_{m\in \B^d}$  span $\R^{n-1}$, and so a necessary condition for \eqref{eq:span} to hold is that $2^d \geq n-1$.
\end{remark}
 
\subsection{Overview and relation to prior work}
\label{sec:overview}

We give here a high level overview of the ideas behind the proofs of Theorems \ref{thm:main1} and \ref{thm:main2}, and compare to the previous work of the first author and R.V.~Kohn \cite{drenska2019PDE}.

We show in Proposition \ref{prop:dppuN} that the rescaled value function $u_N$ satisfies the dynamic programming principle
\[u_N(x,t;m) = \min_{|f|\leq 1}\max_{b=\pm 1}u_N(x +   \eps b(q(m) - \one f),t + \eps^2;m|b),\]
where we write $\eps=N^{-\frac12}$ for convenience. The standard way to extract a limiting PDE from a dynamic programming principle is to replace $u_N(x,t;m)$ by a smooth function $u(x,t)$, \emph{independent} of $m$, and Taylor expand the function $u$. Neglecting error terms, this yields
\[u(x,t) = \min_{|f|\leq 1}\max_{b=\pm 1}\left\{u(x,t) + \eps^2 u_t(x,t) + b\eps \langle\nabla u(x,t),\delta\rangle + \frac{\eps^2}{2}\left\langle \nabla^2u(x,t)\delta,\delta\right\rangle \right\},\]
where $\delta = q(m) - \one f$.  To simplify the discussion, let us assume the translation property \eqref{eq:translation} holds. As in Remark \ref{rem:classical}, this implies that the solution $u$ of \eqref{eq:PDE}, or any candidate for the limit of $u_N$, satisfies $\nabla^2 u\one =0$. This simplifies the dynamic programming principle to read
\[u = \min_{|f|\leq 1}\max_{b=\pm 1}\left\{u + \eps^2 u_t + b\eps \langle\nabla u,q(m)-\one f\rangle + \frac{\eps^2}{2}\left\langle \nabla^2u\,q(m),q(m)\right\rangle \right\},\]
where we have dropped the dependence on $(x,t)$. We can rearrange this to find that
\begin{equation}\label{eq:dpptaylor}
u_t + \min_{|f|\leq 1}\max_{b=\pm 1}\left\{b\eps^{-1} \langle\nabla u,q(m) - \one f\rangle + \frac{1}{2}\left\langle \nabla^2u\, q(m),q(m)\right\rangle \right\} = 0.
\end{equation}
From here, we see that the ``optimal'' choice for the market is $b = \text{sign}(\langle \nabla u,q(m)-\one f\rangle)$ and the ``optimal'' investor strategy is
\begin{equation}\label{eq:naive_optimal}
f = \frac{\langle \nabla u,q(m)\rangle}{\langle \nabla u,\one\rangle}.
\end{equation}
Indeed, this strategy is admissible, i.e., $f\in [-1,1]$, since $q(m)\in [-1,1]^n$ and \eqref{eq:strict_inc} implies $u_{x_i}>0$ for all $i$. In fact, \eqref{eq:naive_optimal} is exactly a weighted average of the expert strategies, weighted by the partial derivatives $u_{x_i}$. This choice sets sets the first term to be zero in the min-max in \eqref{eq:dpptaylor}, which yields
\begin{equation}\label{eq:mPDE}
u_t + \frac{1}{2}\langle \nabla^2 u \, q(m),q(m)\rangle = 0.
\end{equation}
However, this PDE depends on the state $m\in \B^d$ on the de Bruijn graph, and we expect this dependence to drop out as $N\to \infty$. In fact, note that the PDE \eqref{eq:PDElinear} is exactly the average of \eqref{eq:mPDE} over $\B^d$.  This indicates that the investor strategy \eqref{eq:naive_optimal} is not, in fact, optimal. 

To see why \eqref{eq:naive_optimal} is suboptimal, we note that \eqref{eq:mPDE} implies that this investor strategy accumulates regret of $\frac{1}{2}\langle \nabla^2 u \, q(m),q(m)\rangle$ in each step of the game, \emph{independent} of the choice made by the market. Furthermore, by setting the first term in \eqref{eq:dpptaylor} to zero, this strategy gives the market \emph{complete} control over the trajectory of the game on the de Bruijn graph. The market will choose the binary stream $b_1,b_2,\dots,$ so as to traverse cycles on the de Bruijn graph that are most costly, that is, where $\frac{1}{2}\langle \nabla^2 u \, q(m),q(m)\rangle$ is largest. Thus, unless all de Bruijn cycles have the same average cost, the investor has some incentive to slightly modify \eqref{eq:naive_optimal} to counteract the market and limit this behavior. In essence, we were not justified in dropping the state $m$ from the one step dynamic programming principle, and the optimal strategies must take into account more than one step of the game.

In \cite{drenska2019PDE}, the first author and R.V.~Kohn took the \emph{ansatz} that the optimal investor strategy has the form
\begin{equation}\label{eq:ansatz}
f_i = \frac{\langle \nabla u,q(m)\rangle}{\langle \nabla u,\one\rangle} + \eps f^\#_i,
\end{equation}
and looked for correctors $f^\#_i$ that slightly modified \eqref{eq:mPDE} so that all cycles on the de Bruijn graph were equally expensive.  Choosing an $O(\eps)$ perturbation allows $f_i^\#$ to interact directly with the second order $O(\eps^2)$ terms in the Taylor expansion above. The authors of \cite{drenska2019PDE}  showed that the correctors $f^\#_i$ should be chosen as the solution to a particular linear program over the de Bruijn graph with inequality constraints. There are linear programs for both the investor and the market, leading to upper and lower bounds for the value function for $n=2$ and all $d\geq 1$.  When the values of the two linear programs (for the market and investor) coincide, the upper and lower bounds coincide, the strategies are provably optimal, and the authors establish convergence of the value functions. Currently, it is only known that the values coincide for $n=2$ and $d\leq 4$, and this is obtained though explicitly solving the linear programs and checking. The linear programs become exponentially more complicated as $d$ grows, and finding explicit solutions is a challenging open problem for $d\geq 5$. We expect that the investor strategy we identified in \eqref{eq:opt_strategy} is closely related to this linear program, and may provide clues for solving it explicitly for $d \geq 5$.

In this paper, we take an entirely different approach, and in the end, we essentially show that the \emph{ansatz} \eqref{eq:ansatz} is correct for all $n\geq 2$ and $d\geq 1$. We say \emph{essentially} because our optimal strategy (see the discussion in Remark \ref{rem:optimal} and Eq.~\eqref{eq:opt_strategy}) has the form
\[f_i =\frac{\langle \nabla u,q(m)\rangle}{\langle \nabla u,\one\rangle} + O(k\eps),\]
where $k\to \infty$ as $\eps\to 0$. While $k$ can increase to infinity arbitrarily slowly, the optimal value (for the best convergence rate) is $k\sim \eps^{-1/3}$. We compare this with the ansatz \eqref{eq:ansatz}, which implicitly assumes $f_i^\#$ is bounded, independent of $\eps$. It is an open problem to determine if the ansatz \eqref{eq:ansatz} is correct in general, with the sharp $O(\eps)$ perturbation.

Our approach follows more closely to the classical viscosity solutions approach to optimal control. Instead of looking for optimal market and investor strategies and using these to prove convergence of the value function, we focus our attention directly on the value function itself, and use ideas from homogenization theory to show how the value function locally averages out over the de Bruijn graph.  To briefly summarize our approach, instead of taking one step in the dynamic programming principle, we take a large number of steps $k$. This results in the $k$-step dynamic programming principle (proved in Proposition \ref{prop:dppuN})
\[u_N(x,t;m) = \min_{|f_1|\leq 1}\max_{b_1=\pm 1}\cdots \min_{|f_{k}|\leq 1}\max_{b_{k}=\pm 1} u_N\bigg(x +  \eps\sum_{i=1}^{k} b_i(q(m^i) - \one f_i),t + \eps^2k;m^{k+1}\bigg),\]
where $\eps=N^{-1/2}$. We proceed in the same way as above, and replace $u_N$ by a smooth function $u$ and Taylor expand to obtain
\begin{align*}
u(x,t) &= \min_{|f_1|\leq 1}\max_{b_1=\pm 1}\cdots \min_{|f_{k}|\leq 1}\max_{b_{k}=\pm 1}\Big\{u(x,t) + k\eps^2 u_t(x,t) + \eps\sum_{i=1}^kb_i \langle\nabla u(x,t),\delta_i\rangle\\
&\hspace{3.5in} + \frac{\eps^2}{2}\sum_{i,j=1}^kb_ib_j\left\langle \nabla^2u(x,t)\delta_i,\delta_j\right\rangle \Big\},
\end{align*}
where  $m^1=m$ and $m^{i+1}=m^i|b_i$ for $i=1,\dots,k$, and $\delta_i = q(m^i) - \one f_i$. We can rearrange this to find that
\begin{equation}\label{eq:kstepu}
u_t +\frac{1}{k}\min_{|f_1|\leq 1}\max_{b_1=\pm 1}\cdots \min_{|f_{k}|\leq 1}\max_{b_{k}=\pm 1}\left\{ \eps^{-1}\sum_{i=1}^kb_i \langle\nabla u,\delta_i\rangle+ \frac12\sum_{i,j=1}^kb_ib_j\left\langle \nabla^2u\,\delta_i,\delta_j\right\rangle \right\} = 0.
\end{equation}
This allows us to reduce the problem to a repeated two-player game with a quadratic payoff function---the repeated min-max problem in \eqref{eq:kstepu}. We establish asymptotics for the optimal \emph{value} of this game as $k\to \infty$ and $\eps\to 0$, and find that the initial state $m$ averages out of the equation. This allows us to obtain a PDE that is independent of the state $m$, provided that we take $k\to \infty$ as $N\to \infty$.  Along the way, we obtain an asymptotically optimal strategy for the investor, which renders the market indifferent, but this is not directly used in the proofs.

In the previous work \cite{drenska2019PDE}, the authors proved convergence rates of $O(\eps)$ in the context of Theorem \ref{thm:main1} and $O(\eps|\log(\eps)|)$ in the context of Theorem \ref{thm:main2}, for $n=2$ and $d\leq 4$, while also obtaining upper and lower bounds on the value function for $n=2$ and $d \geq 5$. Our convergence rates of $O(\eps^{1/3})$ and $O(\eps^{1/3}|\log(\eps)|)$ are worse, due to the fact that our $k$-step dynamic programming principle \eqref{eq:kstepu} leads to larger errors from Taylor expansion, on the order of $O(k^3\eps^3)$ instead of $O(\eps^3)$, and the fact that we must send $k\to \infty$ as $N\to \infty$ to ensure the initial state $m$ averages out in \eqref{eq:kstepu}. We show in Theorem \ref{thm:cell} that the state $m$ averages out at a rate of $O\left( \frac1k \right)$, and this must be balanced with the Taylor expansion errors to obtain our final convergence rate. It would be interesting to combine our observations of the optimal strategy in \eqref{eq:opt_strategy} with the methods used in \cite{drenska2019PDE} in an attempt to improve the rates in Theorems \ref{thm:main1} and \ref{thm:main2} to match those in \cite{drenska2019PDE} when $d \geq 5$ and $n\geq 3$. We expect this will require some slight modifications to the strategy \eqref{eq:opt_strategy} so that the gradient $p$ and Hessian $X$ are updated at each step of the game, instead of once per $k$-block.

Let us also mention that, at first sight, the PDE \eqref{eq:PDE} and the PDE identified in \cite{drenska2019PDE} appear quite different. We show here that they are equivalent when $n=2$, and thus \eqref{eq:PDE} is the appropriate generalization for $n\geq 3$. When $n=2$, we write $p^\perp = (-p_2,p_1)$ for $p\in \R^2$, and we note that for any $m\in \B^d$ and $p\in \R^2$ with $\langle p,\one\rangle > 0$ we have
\[q(m) - \frac{\langle p,q(m)\rangle}{\langle p,\one\rangle}\one = \frac{q_2(m)-q_1(m)}{\langle p,\one\rangle}p^\perp.\]
Therefore, the equation \eqref{eq:PDE} becomes
\[u_t + C^\#\frac{\langle \nabla^2 u\nabla u^\perp,\nabla u^\perp\rangle}{\langle \nabla u,\one\rangle^2} = 0,\]
where
\[C^\# = \frac{1}{2^{d+1}}\sum_{m\in \B^d}(q_2(m)-q_1(m))^2.\]
This is the same as the PDE identified in \cite{drenska2019PDE} (see, e.g., \cite[Eq.~(5.1)]{drenska2019PDE}), except that in \cite{drenska2019PDE}, the equation is written in the rotated coordinates $(\xi,\eta) = (x_1-x_2,x_1+x_2)$ (we note that the variables $\xi$ and $\eta$ have completely different meanings in our paper, as we do not use the rotated coordinates).

\section{Analysis of the discrete two-player game}
\label{sec:game}

We prove several properties of the discrete game, including monotonicity, translation invariance, and discrete regularity. It will be convenient to extend the concatenation notation defined in \eqref{eq:concat1} to allow for concatenation of longer symbols. We thus use the notation $m|s$ for concatenation of $m\in \B^d$ and $s\in \B^j$, with the result being an element of $\B^d$ ending with $s$. If $j < d$ then 
\[m|s = (m_{j+1},m_{j+2},\dots,m_d,s_1,s_2,\dots,s_j),\]
and if $j \geq d$ then
\[m|s = (s_{j-d+1},s_{j-d+2},\dots,s_j).\]
The notation $m|s|b$ means $(m|s)|b$, and so on. For simplicity we write $m_{+}:=m|1$ and $m_-:=m|-1$. We note that $m|s$ is exactly the state arrived at by starting at node $m$ on the de Bruijn graph and following the edges defined by $s_1,s_2,\dots,s_j$.

A number of properties of the value function $V_N$ follow directly from Definition \ref{def:value}. 
\begin{lemma}\label{lem:MSTI}
Let $N\geq 1$, $1 \leq \ell \leq N$, and $m\in \B^d$.
The following hold.
\begin{enumerate}[(i)]
\item If \eqref{eq:strict_inc} holds, then for all $x\in \R^n$ and $v\in [0,\infty)^n$ we have
\[V_N(x+v, \ell;m) \geq V_N(x,\ell;m) + \cg\langle v,\one\rangle.\]
\item If \eqref{eq:translation} holds, then for all $x\in \R^n$ and $t>0$  
\[V_N (x +t\one,\ell; m) = V_N(x,\ell; m) + t.\]
\item If $g$ is Lipschitz continuous then for all $x,y\in \R^n$ we have
\[|V_N(x,\ell;m) - V(y,\ell;m)| \leq \Lip(g) |x-y|.\]
\end{enumerate}
\end{lemma}
\begin{proof}
The proofs of (i) and (ii) follow directly from Definition \ref{def:value}. For (iii) we have
\begin{align*}
V_N(x,\ell;m) &= \min_{|f_{\ell}|\leq 1}\max_{b_{\ell}=\pm 1}\cdots \min_{|f_{N-1}|\leq 1}\max_{b_{N-1}=\pm 1}g\left( x + \sum_{i=\ell}^{N-1} b_i(q(m^i) - f_i\one) \right)\\
&\leq \min_{|f_{\ell}|\leq 1}\max_{b_{\ell}=\pm 1}\cdots \min_{|f_{N-1}|\leq 1}\max_{b_{N-1}=\pm 1}\left[g\left( y + \sum_{i=\ell}^{N-1} b_i(q(m^i) - f_i\one) \right) + \Lip(g)|x-y|\right]\\
&=V_N(y,\ell;m) + \Lip(g)|x-y|,
\end{align*}
which completes the proof.
\end{proof}

A key property of the value function is the dynamic programming principle. We record below a $k$-step version for $V_N$.
\begin{proposition}[Dynamic Programming Principle]\label{prop:dpp}
For any $N\geq 1$,  $x\in \R^n$, $m\in \B^d$, $k\geq 1$ and $\ell\leq N-k$ it holds that
\begin{equation*}
V_N(x,\ell;m) = \min_{|f_1|\leq 1}\max_{b_1=\pm 1}\cdots \min_{|f_{k}|\leq 1}\max_{b_{k}=\pm 1} V_N\left(x +  \sum_{i=1}^{k} b_i(q(m^i) - \one f_i),\ell + k;m^{k+1}\right),
\end{equation*}
where $m^1 = m$ and $m^{i+1} = m^i|b_i$ for $i=1,\dots,k$.
\end{proposition}
\begin{proof}
By Definition \ref{def:value} we have
\[V_N(x,\ell;m) = \min_{|f_{\ell}|\leq 1}\max_{b_{\ell}=\pm 1}\cdots \min_{|f_{N-1}|\leq 1}\max_{b_{N-1}=\pm 1}g\left( x + \textstyle\sum_{i=\ell}^{N-1} b_i(q(\tilde{m}^i) - f_i\one) \right)\]
where $\tilde{m}^\ell=m$ and $\tilde{m}^{i+1}=\tilde{m}^i|b_i$ for $i=\ell,\dots,N-1$. Noting that
\begin{align*}
&V_N\left(x + \textstyle \sum_{i=\ell}^{\ell+k-1} b_i(q(m^i) - f_i\one),\ell+k;\tilde{m}^{\ell+k}\right) \\
&\hspace{1in}=\min_{|f_{k+\ell}|\leq 1}\max_{b_{k+\ell}=\pm 1}\cdots \min_{|f_{N-1}|\leq 1}\max_{b_{N-1}=\pm 1}g\left( x + \textstyle \sum_{i=k}^{N-1} b_i(q(\tilde{m}^i) - f_i\one) \right),
\end{align*}
we have that
\[V_N(x,\ell;m)= \min_{|f_{\ell}|\leq 1}\max_{b_{\ell}=\pm 1}\cdots \min_{|f_{\ell+k-1}|\leq 1}\max_{b_{\ell+k-1}=\pm 1}V_N\left(x + \textstyle \sum_{i=\ell}^{\ell+k-1} b_i(q(\tilde{m}^i) - f_i\one),\ell+k;\tilde{m}^{\ell+k}\right).\]
Re-indexing $i$ we have
\[V_N(x,\ell;m) = \min_{|f_1|\leq 1}\max_{b_1=\pm 1}\cdots \min_{|f_{k}|\leq 1}\max_{b_{k}=\pm 1} V_N\left(x + \textstyle \sum_{i=1}^{k} b_i(q(m^i) - \one f_i),\ell + k;m^{k+1}\right),\]
where $m^1=m$ and $m^{i+1}=m^i|b_i$ for $i=1,\dots,k$, which completes the proof.
\end{proof}

We immediately obtain a dynamic programming principle for the rescaled value function $u_N$ defined in \eqref{eq:uN}.
\begin{proposition}[Rescaled Dynamic Programming Principle]\label{prop:dppuN}
For  $N\geq 1$, $m\in \B^d$, $k\geq 1$ and $0\leq t \leq 1-N^{-1}k$, it holds that
\begin{equation*}
u_N(x,t;m) = \min_{|f_1|\leq 1}\max_{b_1=\pm 1}\cdots \min_{|f_{k}|\leq 1}\max_{b_{k}=\pm 1} u_N\bigg(x +  N^{-\frac12}\sum_{i=1}^{k} b_i(q(m^i) - \one f_i),t + N^{-1}k;m^{k+1}\bigg).
\end{equation*}
where $m^1 = m$ and $m^{i+1} = m^i|b_i$ for $i=1,\dots,k$.
\end{proposition}
\begin{proof}
By the definition of $u_N$ \eqref{eq:uN} we have
\[V_N(x,\ell;m) = \sqrt{N}u_N(N^{-1/2}x,N^{-1}\ell;m).\]
By Proposition \ref{prop:dpp} we thus have
\begin{align*}
u_N(N^{-1/2}x,N^{-1}\ell;m) =& \min_{|f_1|\leq 1}\max_{b_1=\pm 1}\cdots \min_{|f_{k}|\leq 1}\max_{b_{k}=\pm 1} \\
&\hspace{3mm}u_N\left(N^{-1/2}x + \textstyle N^{-1/2}\sum_{i=1}^{k} b_i(q(m^i) - \one f_i),N^{-1}\ell + N^{-1}k;m^{k+1}\right)
\end{align*}
for any $x\in \R^n$, $m\in \B^d$, $k\geq 1$ and $\ell\leq N-k$, where $m^i$ are given as in Proposition \ref{prop:dpp}. Setting $y=N^{-1/2}x$ and $t=N^{-1}\ell$ we obtain
\begin{align*}
u_N(y,t;m) =& \min_{|f_1|\leq 1}\max_{b_1=\pm 1}\cdots \min_{|f_{k}|\leq 1}\max_{b_{k}=\pm 1} u_N\bigg(y +  N^{-\frac12}\sum_{i=1}^{k} b_i(q(m^i) - \one f_i),t + N^{-1}k;m^{k+1}\bigg).
\end{align*}
Since this also holds for any $t\in [0,1]$ with $\lceil Nt \rceil = \ell$, the proof is complete.
\end{proof}
\begin{remark}
Notice that the dynamic programming principle given in Propositions \ref{prop:dpp} and \ref{prop:dppuN} are coupled systems of $2^d$ equations involving all $2^d$ value functions. In particular, the states $m$ on the left hand side and $m^{k+1}$ on the right hand side, are in general \emph{different} states on the de Bruijn graph. This causes some difficulties with obtaining a Hamilton-Jacobi-Isaacs equation directly from the dynamic programming principle, and is the reason we consider a $k$-step dynamic programming principle, instead of the usual $1$-step dynamic programming principle. As we show in Section \ref{sec:cell} below, when $k$ is large enough, the initial state $m$ in the $k$-step dynamic programming principle is \emph{forgotten} (it averages out over the de Bruijn graph), and this allows us to decouple the $2^d$ dynamic programming principle equations into a single averaged equation. 
\label{rem:dpp}
\end{remark}

\section{The local problem}
\label{sec:cell}

We now study a local problem that arises from the $k$-step dynamic programming principle identified in Proposition \ref{prop:dppuN}. We make the following definition.

\begin{definition}[Local problem]\label{def:localproblem}
Let $X\in \S(n)$\footnote{$\S(n)$ denotes the space of $n\times n$ real symmetric matrices.}, $p\in \R^n$, $k \geq 1$, $\eps>0$, and $m\in \B^d$. The \emph{local problem} is given by
\begin{equation}\label{eq:Cell}
\L_{k,\eps}(X,p,m)=\min_{|f_1|\leq 1}\max_{b_1=\pm 1}\cdots \min_{|f_k|\leq 1}\max_{b_k=\pm 1}\left\{\eps^{-1} \sum_{i=1}^kb_i\langle p,\delta_i\rangle  + \frac12\sum_{i,j=1}^kb_ib_j\langle X\delta_i,\delta_j\rangle\right\},
\end{equation}
where  $m^1=m$ and $m^{i+1}=m^i|b_i$ for $i=1,\dots,k$, and 
\begin{equation}\label{eq:deltai}
\delta_i = q(m^i) - \one f_i.
\end{equation}
We will write $\L_{k,\eps}$ in place of $\L_{k,\eps}(X,p,m)$ when the values of $X,p$ and $m$ are clear from context.
\end{definition}
The motivation for the local problem was given in Section \ref{sec:overview}. In particular, the local problem is the main operator appearing in \eqref{eq:kstepu} with $p=\nabla u$ and $X=\nabla^2 u$, and so \eqref{eq:kstepu} can be written as
\[u_t + \frac{1}{k}\L_{k,\eps}(\nabla^2u,\nabla u,m) = 0.\]
We show in this section that the initial state $m$ averages out of the local problem $\L_{k,\eps}(X,p,m)$ as $k\to \infty$ at a rate of $O\left( \frac{1}{k} \right)$. This allows us to obtain a PDE that is independent of the initial state $m$. The situation is similar to how small scale oscillations in a cell problem average out in homogenization theory. In fact, the local problem is very much analogous to a cell problem from homogenization, except that the oscillations in the local problem occur in an auxiliary variable living on a discrete graph (the de Bruijn graph). 

Our main result in this section is the following convergence rate for the local problem.
\begin{theorem}[Local problem]\label{thm:cell}
Assume \eqref{eq:redund} holds. Let $X\in \S(n)$, $p\in (0,\infty)^n$, $m\in \B^d$, $k\geq d+1$, $\eps>0$, and set $\gamma_p = \min_{1\leq i \leq n}p_i$. Then there exists $C,c>0$, depending only on $n$, such that whenever $\|X\|k\eps \leq  c\,\cq \gamma_p$ we have
\begin{equation}\label{eq:Cellrate}
\left|\frac{1}{k}\L_{k,\eps}(X,p,m) -\frac{1}{2^{d+1}}\sum_{\eta \in Q(p)} \langle X\eta,\eta\rangle\right|\leq C\|X\|\left(\frac{d}{k} + \|X\| \gamma_p^{-1}k\eps\right).
\end{equation}
\end{theorem}
Here, $\|X\|$ is the operator norm of $X$ given by 
\[\|X\| = \sup\{|X\eta| \, : \, \eta \in \R^n \text{ and }|\eta|=1\}.\]
We also recall $\cq>0$ is defined in \eqref{eq:thetaq_def}. 

The remainder of this section is devoted to proving Theorem \ref{thm:cell}. For this, we require some additional notation. For $p\in \R^n$ and $m\in \B^n$ we define 
\begin{equation}\label{eq:etadef}
\xi(p,m) = q(m) - \frac{\langle p, q(m)\rangle}{\langle p, \one\rangle}\one.
\end{equation}
For $p\in \R^n$, $X\in \S(n)$ and $m\in \B^n$, we define $\H_k(X,p,m)$ by $\H_0(X,p,m)=0$ and 
\begin{equation}\label{eq:Hm}
\H_{k}(X,p,m) = \frac{1}{2}\langle X\xi(p,m),\xi(p,m)\rangle + \frac{1}{2}\left( \H_{k-1}(X,p,m_+) + \H_{k-1}(X,p,m_-) \right)
\end{equation}
for $k\geq 1$. We will often make the dependence on $X$ and $p$ implicit and write $\xi(m) = \xi(p,m)$ and $\H_k(m)=\H_k(X,p,m)$, to reduce the notational burden. Notice that 
\[-\langle p,\one\rangle \leq \langle p,q(m)\rangle \leq \langle p,\one\rangle,\]
and so $\xi(m) \in [-2,2]^n$. This implies that $|\xi(m)|\leq 2\sqrt{n}$.

We record some important properties of $\H_k(m)$.
\begin{proposition}\label{prop:S}
There exists $C>0$ depending only on $n$ such that the following hold.
\begin{enumerate}[(i)]
\item For all $m\in \B^d$, $\H_1(m)=\frac{1}{2}\langle X\xi(m),\xi(m)\rangle$ and for $k\geq 2$
\begin{equation*}
\H_k(m) =\frac{1}{2}\langle X\xi(m),\xi(m)\rangle+ \sum_{\ell=1}^{k-1} \frac{1}{2^{\ell+1}}\sum_{s\in \B^\ell} \langle X\xi(m|s),\xi(m|s)\rangle.
\end{equation*}
\item For all $m\in \B^d$ and $k\geq 0$ we have
\begin{equation*}
|\H_k(X,p,m_+) - \H_k(X,p,m_-)|\leq Cd\|X\|.
\end{equation*}
\item For all $m\in \B^d$ and $k\geq d+1$ we have
\[\left| \H_k(X,p,m) - \frac{k}{2^{d+1}}\sum_{\eta \in Q(p)} \langle X\eta,\eta\rangle\right|\leq Cd\|X\|.\]
\end{enumerate}
\end{proposition}
\begin{remark}
\label{rem:Hm}
Proposition \ref{prop:S} (i) shows that $\H_k(m)$ is exactly a weighted average of the quantities  $\zeta(m):=\frac{1}{2}\langle X\xi(m),\xi(m)\rangle$ over a de Bruijn tree of depth $k$ rooted at $m$. See Figure \ref{fig:tree} for an illustration. In the figure we replaced $-1$ with $0$ for convenience. The statements (ii) and (iii) follow from the fact that only the first $d$ layers of the tree depend on the root node $m$, and so the root node averages out when $k\gg d$. Furthermore, all contributions to the difference $\H_k(m_+) - \H_k(m_-)$ of depth $d+1$ and higher exactly cancel out.
\end{remark}

\begin{figure}
\begin{center}
\begin{tikzpicture}[yscale=1.2]
\node[left] at (0,0) {101};
\draw[thick,->] (0,0)--(1,1);
\draw[thick,->] (0,0)--(1,-1);
\node[right] at (1,1) {01\red{1}};
\node[right] at (1,-1) {01\red{0}};
\draw[thick,->] (1.7,1)--(2.7,1.5);
\draw[thick,->] (1.7,1)--(2.7,0.5);
\draw[thick,->] (1.7,-1)--(2.7,-1.5);
\draw[thick,->] (1.7,-1)--(2.7,-0.5);
\node[right] at (2.7,1.5) {11\red{1}};
\node[right] at (2.7,0.5) {11\red{0}};
\node[right] at (2.7,-1.5) {10\red{0}};
\node[right] at (2.7,-0.5) {10\red{1}};
\draw[thick,->] (3.4,1.5)--(4.4,1.75);
\draw[thick,->] (3.4,1.5)--(4.4,1.25);
\draw[thick,->] (3.4,0.5)--(4.4,0.75);
\draw[thick,->] (3.4,0.5)--(4.4,0.25);
\draw[thick,->] (3.4,-1.5)--(4.4,-1.75);
\draw[thick,->] (3.4,-1.5)--(4.4,-1.25);
\draw[thick,->] (3.4,-0.5)--(4.4,-0.75);
\draw[thick,->] (3.4,-0.5)--(4.4,-0.25);
\node[right] at (4.4,1.75) {11\red{1}};
\node[right] at (4.4,1.25) {11\red{0}};
\node[right] at (4.4,0.75) {10\red{1}};
\node[right] at (4.4,0.25) {10\red{0}};
\node[right] at (4.4,-1.75) {00\red{0}};
\node[right] at (4.4,-1.25) {00\red{1}};
\node[right] at (4.4,-0.75) {01\red{0}};
\node[right] at (4.4,-0.25) {01\red{1}};
\draw[thick,->] (5.1,1.75)--(6.1,1.75+0.125);
\draw[thick,->] (5.1,1.75)--(6.1,1.75-0.125);
\draw[thick,->] (5.1,1.25)--(6.1,1.25+0.125);
\draw[thick,->] (5.1,1.25)--(6.1,1.25-0.125);
\draw[thick,->] (5.1,0.75)--(6.1,0.75+0.125);
\draw[thick,->] (5.1,0.75)--(6.1,0.75-0.125);
\draw[thick,->] (5.1,0.25)--(6.1,0.25+0.125);
\draw[thick,->] (5.1,0.25)--(6.1,0.25-0.125);
\draw[thick,->] (5.1,-1.75)--(6.1,-1.75+0.125);
\draw[thick,->] (5.1,-1.75)--(6.1,-1.75-0.125);
\draw[thick,->] (5.1,-1.25)--(6.1,-1.25+0.125);
\draw[thick,->] (5.1,-1.25)--(6.1,-1.25-0.125);
\draw[thick,->] (5.1,-0.75)--(6.1,-0.75+0.125);
\draw[thick,->] (5.1,-0.75)--(6.1,-0.75-0.125);
\draw[thick,->] (5.1,-0.25)--(6.1,-0.25+0.125);
\draw[thick,->] (5.1,-0.25)--(6.1,-0.25-0.125);
\node[right] at (6.1,1.75+0.125) {$\cdots$};
\node[right] at (6.1,1.75-0.125) {$\cdots$};
\node[right] at (6.1,1.25+0.125) {$\cdots$};
\node[right] at (6.1,1.25-0.125) {$\cdots$};
\node[right] at (6.1,0.75+0.125) {$\cdots$};
\node[right] at (6.1,0.75-0.125) {$\cdots$};
\node[right] at (6.1,0.25+0.125) {$\cdots$};
\node[right] at (6.1,0.25-0.125) {$\cdots$};
\node[right] at (6.1,-1.75+0.125){$\cdots$};
\node[right] at (6.1,-1.75-0.125){$\cdots$};
\node[right] at (6.1,-1.25+0.125){$\cdots$};
\node[right] at (6.1,-1.25-0.125){$\cdots$};
\node[right] at (6.1,-0.75+0.125){$\cdots$};
\node[right] at (6.1,-0.75-0.125){$\cdots$};
\node[right] at (6.1,-0.25+0.125){$\cdots$};
\node[right] at (6.1,-0.25-0.125){$\cdots$};
\draw[dashed] (-0.35,-0.15)--(-0.35,-3.5);
\node[right] at (-0.4,-3.5) {$\zeta(\text{101})$};
\draw[dashed] (1.35,-1.15)--(1.35,-3.1);
\node[right] at (1.3,-3.1) {$+\frac{1}{2}(\zeta(\text{011}) + \zeta(\text{010}))$};
\draw[dashed] (3.05,-1.65)--(3.05,-2.7);
\node[right] at (3,-2.7) {$+\frac{1}{4}(\zeta(\text{111}) + \zeta(\text{110})+\zeta(\text{101}) + \zeta(\text{100}))$};
\draw[dashed] (4.75,-1.9)--(4.75,-2.3);
\node[right] at (4.6,-2.3) {$ +\frac{1}{8}(\zeta(\text{111}) + ...)$};
\end{tikzpicture}
\end{center}
\caption{Summing over a de Bruijn tree}
\label{fig:tree}
\end{figure}
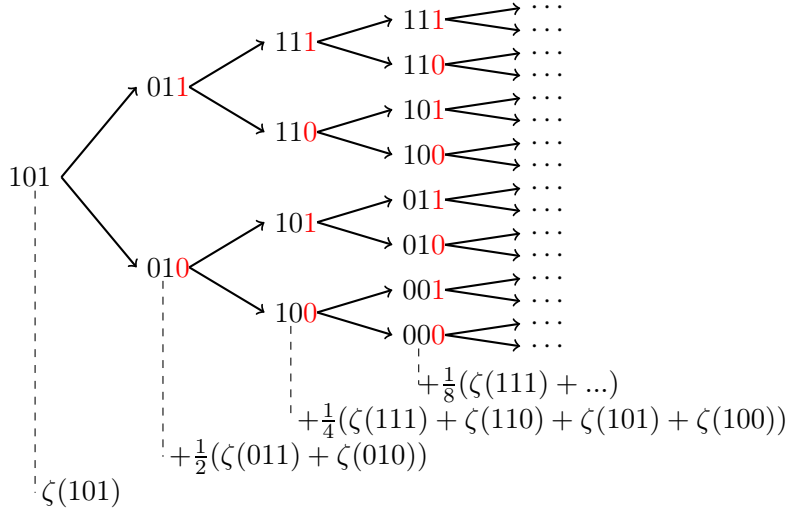

\begin{proof}
We first prove (i). Define $\tilde{\H}_k$ by $\tilde{\H}_1(m) = \frac{1}{2}\langle X\xi(m),\xi(m)\rangle$ and
\[\tilde{\H}_k(m) :=\frac{1}{2}\langle X\xi(m),\xi(m)\rangle+ \sum_{\ell=1}^{k-1} \frac{1}{2^{\ell+1}}\sum_{s\in \B^\ell} \langle X\xi(m|s),\xi(m|s)\rangle.\]
We will show that $\tilde{\H}_k$ satisfies the recursion \eqref{eq:Hm}, and so $\H_k=\tilde{\H}_k$. It is clear that \eqref{eq:Hm} holds for $k=2$, so we may assume $k\geq 3$. Then we compute
\begin{align*}
\frac{1}{2}\left( \tilde{\H}_{k-1}(m_+) + \tilde{\H}_{k-1}(m_-) \right)&=\frac{1}{4}\left(\langle X\xi(m_+),\xi(m_+)\rangle+\langle X\xi(m_-),\xi(m_-)\rangle  \right)\\
&\hspace{4mm}+\frac{1}{2}\sum_{\ell=1}^{k-2} \frac{1}{2^{\ell+1}}\sum_{s\in \B^\ell} \left(\langle X\xi(m_+|s),\xi(m_+|s)\rangle+ \langle X\xi(m_-|s),\xi(m_-|s)\rangle\right)\\
&=\frac{1}{4}\sum_{s\in \B^1}\langle X\xi(m|s),\xi(m|s)\rangle+\sum_{\ell=1}^{k-2} \frac{1}{2^{\ell+2}}\sum_{s\in \B^{\ell+1}} \langle X\xi(m|s),\xi(m|s)\rangle\\
&=\sum_{\ell=1}^{k-1} \frac{1}{2^{\ell+1}} \sum_{s\in \B^{\ell}} \langle X\xi(m|s),\xi(m|s)\rangle\\
&=\tilde{\H}_k(m) - \frac{1}{2}\langle X\xi(m),\xi(m)\rangle,
\end{align*}
which completes the proof of (i).

To prove (ii), we note that $m_+|s = m_-|s$ for $s\in \B^\ell$ with $\ell \geq d$.  Therefore, we have
\begin{align*}
\H_k(m_+) - \H_k(m_-)&=\frac{1}{2}\left(\langle X\xi(m_+),\xi(m_+)\rangle-\langle X\xi(m_-),\xi(m_-)\rangle\right)\\
&\hspace{0.25in}+ \hspace{-4mm}\sum_{\ell=1}^{\min\{k-1,d\}}\hspace{-2mm} \frac{1}{2^{\ell+2}}\sum_{s\in \B^{\ell}}\left( \langle X\xi(m_+|s),\xi(m_+|s)\rangle -\langle X\xi(m_-|s),\xi(m_-|s)\rangle\right).
\end{align*}
Therefore, there exists $C$, depending only on $n$, such that for all $m\in \B^d$ and $k\geq 0$ we have $|\H_k(m_+) - \H_k(m_-)|\leq Cd\|X\|$.

To prove (iii), we note that for $s\in \B^\ell$ with $\ell\geq d$
\[m|s = (s_{\ell-d+1},\dots,s_{\ell})\]
is independent of $m$. Therefore for $k\geq d+1$ we have
\begin{align*}
\H_k(m)&=\frac{1}{2}\langle X\xi(m),\xi(m)\rangle+ \sum_{\ell=1}^{d-1} \frac{1}{2^{\ell+1}}\sum_{s\in \B^\ell} \langle X\xi(m|s),\xi(m|s)\rangle+\sum_{\ell=d}^{k-1} \frac{2^{\ell-d}}{2^{\ell+1}}\sum_{s\in \B^d} \langle X\xi(s),\xi(s)\rangle\\
&= \frac{k-d}{2^{d+1}}\sum_{s\in \B^d} \langle X\xi(s),\xi(s)\rangle + O(d\|X\|)\\
&= \frac{k}{2^{d+1}}\sum_{\eta \in Q(p)} \langle X\eta,\eta\rangle + O(d\|X\|),
\end{align*}
which completes the proof.
\end{proof}

A main technical tool in the proof of Theorem \ref{thm:cell} is the computation of the exact optimality conditions for one step of the min-max problem.
\begin{lemma}\label{lem:opt}
Let $\eps>0$, $S:\B\to \R$, and let $h_1,h_2:[-1,1]\to \R$ be smooth. Consider the min-max problem
\begin{equation}\label{eq:minmax}
M:=\min_{|f|\leq 1}\max_{b=\pm 1}\left\{ bh_1(f) + \eps(S(b) + h_2(f)) \right\}.
\end{equation}
Assume that 
\begin{equation}\label{eq:h11}
h_1(-1) > \frac{\eps}{2}(S(-1)-S(1))) > h_1(1),
\end{equation}
and 
\begin{equation}\label{eq:h12}
h_1'(f) + \eps |h_2'(f)| < 0 \ \ \ \text{ for all }f\in [-1,1].
\end{equation}
Then \eqref{eq:minmax} is minimized by $f^*\in [-1,1]$ satisfying
\begin{equation}\label{eq:f}
h_1(f^*) =\frac{\eps}{2}(S(-1)-S(1))
\end{equation}
and the optimal value of the min-max problem is
\begin{equation}\label{eq:L}
M =\eps h_2(f^*) + \frac{\eps}{2}(S(1) + S(-1)).
\end{equation}
\end{lemma}
\begin{proof}
Write 
\[M_+(f) = h_1(f) + \eps(S(1) + h_2(f)),\]
and
\[M_-(f) = -h_1(f) + \eps(S(-1) + h_2(f)).\]
Then 
\[M = \min_{|f|\leq 1}\max \left\{ M_+(f),M_-(f) \right\}.\]
Since $M_+$ and $M_-$ are continuous, the minimum is attained at some $f^*\in [-1,1]$. 

We claim that $M_+(f^*) = M_-(f^*)$, from which \eqref{eq:f} and \eqref{eq:L} immediately follow. The proof of the claim is split into two steps.

1. We first show that $M_+(f^*)\leq M_-(f^*)$.  Assume to the contrary that $M_+(f^*) > M_-(f^*)$. We first observe that
\[M_-(1) - M_+(1) = -2h_1(1) - \eps(S(1)-S(-1)) > 0\]
due to \eqref{eq:h11}. Therefore $f^* < 1$. Now, note that
\[M_+'(f) = h'(f) + \eps h_2'(f) < 0\]
for all $f\in [-1,1]$, due to \eqref{eq:h12}. Thus, there exists $\delta>0$, sufficiently small, so that $M_+(f^*+\delta) < M_+(f^*)$ and $M_-(f^*+\delta) < M_+(f^*+\delta)$ (by continuity). It follows that
\[\max\{M_+(f^*+\delta),M_-(f^*+\delta)\} = M_+(f^*+\delta) < M_+(f^*) = \max\{M_+(f^*),M_-(f^*)\},\]
which contradicts the minimality of $f^*$.

2. We now show that $M_+(f^*)\geq M_-(f^*)$.  As before, assume to the contrary that $M_+(f^*) < M_-(f^*)$, and observe that
\[M_+(-1) - M_-(-1) = 2h_1(-1) - \eps(S(-1)-S(1)) > 0,\]
due to \eqref{eq:h11}. Therefore $f^* > -1$. Now, note that
\[M_-'(f) \leq -h_1'(f) + \eps h_2'(f) > 0\]
for all $f \in [-1,1]$, due to \eqref{eq:h12}. Thus, for small $\delta>0$ we have $M_-(f^* - \delta) < M_-(f^*)$ and $M_+(f^*-\delta) < M_-(f^*-\delta)$. It follows that
\[\max\{M_+(f^*-\delta),M_-(f^*-\delta)\} = M_-(f^*-\delta) < M_-(f^*) = \max\{M_+(f^*),M_-(f^*)\},\]
which contradicts the minimality of $f^*$. 
\end{proof}

Finally, we require a technical proposition.
\begin{proposition}\label{prop:thetaq}
Assume \eqref{eq:redund} holds and let $p\in (0,\infty)^n$. Then 
\begin{equation}\label{eq:thetaq}
\langle p,\one\rangle - |\langle p,q(m)\rangle| \geq \cq\gamma_p \ \ \ \text{for all }m\in \B^d,
\end{equation}
where $\gamma_p=\min_{1\leq i \leq n}p_i$ and $\cq>0$ is given in \eqref{eq:thetaq_def}.
\end{proposition}
\begin{proof}
Note that
\[-\sum_{i=1}^n p_i q_i(m)_- \leq \langle p,q(m)\rangle \leq \sum_{i=1}^n p_i q_i(m)_+,\]
where $a_+=\max\{a,0\}$ and $a_-=-\min\{a,0\}$. Therefore
\[|\langle p,q(m)\rangle| \leq \max\left\{\sum_{i=1}^n p_i q_i(m)_+,\sum_{i=1}^n p_i q_i(m)_- \right\},\]
and we have
\begin{align*}
\langle p,\one\rangle - |\langle p,q(m)\rangle|&\geq \min\left\{\sum_{i=1}^n p_i (1-q_i(m)_+),\sum_{i=1}^n p_i (1-q_i(m)_-) \right\}\\
&\geq \gamma_p\min\left\{\sum_{i=1}^n (1-q_i(m)_+),\sum_{i=1}^n (1-q_i(m)_-) \right\}= \gamma_p \cq,
\end{align*}
which completes the proof.
\end{proof}

The following lemma shows that the cell problem $\L_{k,\eps}$ is well-approximated by $\H_k$, and essentially completes the proof of Theorem \ref{thm:cell}.
\begin{lemma}\label{lem:cell}
Assume \eqref{eq:redund} holds. Let $X\in \S(n)$, $p\in (0,\infty)^n$, $m\in \B^d$, $k\geq 1$, $\eps>0$, and set $\gamma_p = \min_{1\leq i \leq n}p_i$. Then there exists $C,c>0$, depending only on $n$, such that whenever $\|X\|(k+d)\eps \leq  c\,\cq \gamma_p$ we have
\begin{equation}\label{eq:cell}
|\L_{k,\eps}(X,p,m) -\H_k(X,p,m)| \leq  C\|X\|^2 \gamma_p^{-1}(k+d)k\eps.
\end{equation}
\end{lemma}
\begin{proof}
Recall that $\delta_i = q(m^i) -\one f_i$. We claim that for every $\ell=0,\dots,k$ we have
\begin{align}\label{eq:toinduct}
\L_{k,\eps}\hspace{-1mm}&=\hspace{-1mm}\min_{|f_1|\leq 1}\max_{b_1=\pm 1}\cdots \min_{|f_{k-\ell}|\leq 1}\max_{b_{k-\ell}=\pm 1}\hspace{-1mm}\left\{\eps^{-1} \sum_{i=1}^{k-\ell}b_i\langle p,\delta_i\rangle  + \frac12\sum_{i,j=1}^{k-\ell}b_ib_j\langle X\delta_i,\delta_j\rangle + \H_\ell(m^{k-\ell+1})\right\}\\
&\hspace{4in} +O(\|X\|^2 \gamma_p^{-1}(k+d)\ell\eps),\notag
\end{align}
where when $\ell=k$, the statement reduces to
\[\L_{k,\eps} = \H_\ell(m) + O(\|X\|^2 \gamma_p^{-1}(k+d)k\eps),\]
which completes the proof of the theorem. 

We prove \eqref{eq:toinduct} by induction. The base case of $\ell=0$ is given by the definition of the local problem \eqref{eq:Cell}, since $\H_0(m)=0$ for all $m\in \B^d$. For the inductive step, let us assume \eqref{eq:toinduct} is true for some $\ell \in \{0,\dots,k-1\}$. Then we can write
\begin{align}\label{eq:induct}
\L_{k,\eps}&=\min_{|f_1|\leq 1}\max_{b_1=\pm 1}\cdots \min_{|f_{k-\ell-1}|\leq 1}\max_{b_{k-\ell-1}=\pm 1}\left\{\eps^{-1} \sum_{i=1}^{k-\ell-1}b_i\langle p,\delta_i\rangle  + \frac12\sum_{i,j=1}^{k-\ell-1}b_ib_j\langle X\delta_i,\delta_j\rangle + s_\ell\right\}\\
&\hspace{4in} +O(\|X\|^2 \gamma_p^{-1}(k+d)\ell\eps),\notag
\end{align}
where
\begin{equation*}
s_\ell =\eps^{-1}\hspace{-2mm} \min_{|f_{k-\ell}|\leq 1}\max_{b_{k-\ell}=\pm 1}\left\{b_{k-\ell}\langle p,\delta_{k-\ell}\rangle + \eps\sum_{i=1}^{k-\ell-1}b_{k-\ell}b_i\langle X\delta_{k-\ell},\delta_i\rangle + \frac{\eps}{2}\langle X\delta_{k-\ell},\delta_{k-\ell}\rangle + \H_{\ell}(m^{k-\ell+1}) \right\}.
\end{equation*}
Note that if $\ell=k-1$, then there are no min-max terms nor summations in \eqref{eq:induct} and $\L_{k,\eps}=s_{k-1}$. Similarly, there is no summation term in $s_\ell$ when $\ell=k-1$.

Recall that $m^{k-\ell+1}=m^{k-\ell}|b_{k-\ell}$. Hence, we will apply Lemma \ref{lem:opt} with $S(b_{k-\ell}) = \H_{\ell}(m^{k-\ell}|b_{k-\ell})$,
\[h_1(f_{k-\ell}) = \langle p,\delta_{k-\ell}\rangle + \eps\sum_{i=1}^{k-\ell-1} b_i\langle X\delta_{k-\ell},\delta_i\rangle \ \ \text{ and }\ \  h_2(f_{k-\ell})  = \frac{1}{2}\langle X \delta_{k-\ell},\delta_{k-\ell}\rangle.\]
We need to check conditions \eqref{eq:h11} and \eqref{eq:h12} in Lemma \ref{lem:opt}.
We have by Proposition \ref{prop:thetaq} that
\begin{align*}
h_1(1)&=\langle p,q(m^{k-\ell}) - \one\rangle + \eps\sum_{i=1}^{k-\ell-1} b_i\langle X\delta_{k-\ell},\delta_i\rangle\\
&\leq -\cq\gamma_p + \eps\sum_{i=1}^{k-\ell-1} |X\delta_k| |\delta_i|\\
&\leq -\cq\gamma_p + C\|X\|(k-\ell-1)\eps,
\end{align*}
and
\[h_1(-1)=\langle p,q(m^{k-\ell-1}) + \one\rangle + \eps\sum_{i=1}^{k-\ell-1} b_i\langle X\delta_{k-\ell},\delta_i\rangle \geq \gamma_p \cq - C\|X\|(k-\ell-1)\eps.\]
By Proposition \ref{prop:S} (ii) we have
\[|S(1) - S(-1)| = |\H_{\ell}(m^{k-\ell}_+) - \H_{\ell}(m^{k-\ell}_-)| \leq Cd\|X\|.\]
Thus, to ensure that \eqref{eq:h11} holds we require that 
\begin{equation}\label{eq:cond1}
C\|X\|(k-\ell + d-1)\eps\leq \cq\gamma_p.
\end{equation}
 For \eqref{eq:h12}, note that
\[h_1'(f_{k-\ell}) = -\langle p,\one\rangle - \eps\sum_{i=1}^{k-\ell-1}b_i\langle X\one,\delta_i\rangle, \]
and
\[h_2'(f_{k-\ell}) = -\langle X q(m^{k-\ell}),\one\rangle + f_{k-\ell} \langle X\one,\one\rangle. \]
Therefore
\[h_1'(f_k) \leq -n\gamma_p + C\|X\|(k-\ell-1)\eps \ \ \text{ and }\ \ |h_2'(f_k)|\leq C\|X\|.\]
Since $\cq<1$, we find that \eqref{eq:cond1} is also sufficient for \eqref{eq:h12} to hold, and \eqref{eq:cond1} follows from our assumption that $\|X\|(k+d)\eps \leq  c\,\cq \gamma_p$.

Thus, we can apply Lemma \ref{lem:opt} to find that the optimal $f_{k-\ell}^*$ in the definition of $s_\ell$ satisfies 
\[h_1(f_{k-\ell}^*) = \frac{\eps}{2}(S(-1) - S(1)) = \frac{\eps}{2}(\H_{\ell}(m^{k-\ell}_-) - \H_{\ell}(m^{k-\ell}_+)),\]
and
\begin{equation}\label{eq:sl1}
s_\ell = h_2(f_{k-\ell}^*) + \frac{1}{2}(\H_{\ell}(m^{k-\ell}_+) + \H_{\ell}(m^{k-\ell}_-)).
\end{equation}
Therefore
\begin{equation}\label{eq:optimal_strategy}
f_{k-\ell}^* = \frac{\langle p,q(m^{k-\ell})\rangle + \eps\sum_{i=1}^{k-\ell-1}b_i\langle Xq(m^{k-\ell}),\delta_i\rangle + \frac{\eps}{2}(\H_\ell(m^{k-\ell}_+) - \H_\ell(m^{k-\ell}_-))}{\langle p,\one\rangle + \eps\sum_{i=1}^{k-\ell-1}b_i\langle X\one,\delta_i\rangle}.
\end{equation}
To obtain an asymptotic expression for $f_{k-\ell}^*$, note that by Proposition \ref{prop:S} (ii) we have
\[|h_1(f_{k-\ell}^*)|\leq Cd\|X\|\eps\]
and so
\[|\langle p,q(m^{k-\ell})-\one f_{k-\ell}^*\rangle| \leq C\|X\|(k-\ell+d-1)\eps \leq C\|X\|(k+d).\]
It follows that
\[f_{k-\ell}^* = \frac{\langle p,q(m^{k-\ell})\rangle}{\langle p,\one\rangle} + O(\|X\|\gamma_p^{-1}(k+d)\eps),\]
and so
\begin{align*}
h_2(f_{k-\ell}^*) &= \frac{1}{2}\langle X(q(m^{k-\ell})-\one f_{k-\ell}^*),q(m^{k-\ell})-\one f_{k-\ell}^*\rangle\\
&=\frac{1}{2}\langle X\xi(m^{k-\ell}),\xi(m^{k-\ell})\rangle + O(\|X\|^2\gamma_p^{-1}(k+d)\eps).
\end{align*}
By \eqref{eq:Hm} and \eqref{eq:sl1} we have
\[s_\ell = \H_{\ell+1}(m^{k-\ell}) + O(\|X\|^2\gamma_p^{-1}(k+d)\eps).\]
Inserting this into \eqref{eq:induct} completes the proof by induction.
\end{proof}

\begin{proof}[Proof of Theorem \ref{thm:cell}]
The proof simply combines Proposition \ref{prop:S} (iii) and Lemma \ref{lem:cell}.
\end{proof}

\begin{remark} 
If the translation property \eqref{eq:translation} holds then some of the computations in Lemma \ref{lem:cell} can be simplified. Indeed, in this case Proposition \ref{prop:PDEprop} (iii) shows that the solution $u(x,t)$ of the PDE \eqref{eq:PDE} also satisfies the translation property
\[u(x+s\one,t) = u(x,t) + s.\]
It follows that $\langle \nabla u(x,t), \one\rangle  = 1$ and thus $\nabla^2 u(x,t)\one = 0$. Thus, we may restrict attention in the local problem to Hessians $X\in \S(n)$ that satisfy $X\one=0$. Therefore the local problem \eqref{eq:Cell} becomes
\[\L_{k,\eps}=\min_{|f_1|\leq 1}\max_{b_1=\pm 1}\cdots \min_{|f_k|\leq 1}\max_{b_k=\pm 1}\left\{\eps^{-1} \sum_{i=1}^kb_i\langle p,\delta_i\rangle  + \frac12\sum_{i,j=1}^kb_ib_j\langle Xq(m^i),q(m^j)\rangle\right\}.\]
In particular, the optimization over $f_i$ concerns the linear term only and the proof simplifies greatly. In this case, Theorem \ref{thm:cell} simplifies to read
\[\left|\L_{k,\eps}(X,p,m) -\frac{k}{2^{d+1}}\sum_{m\in \B^d} \langle Xq(m),q(m)\rangle\right|\leq Cd\|X\|.\]
\end{remark}

\section{Analysis of the continuum PDE}
\label{sec:pde}

In this section we analyze the continuum PDE \eqref{eq:PDE}. In particular, we show that under relatively few assumptions, the equation enjoys the comparison principle and has a unique viscosity solution. Under additional assumptions on the expert strategies and the payoff, we furthermore show that the viscosity solution is smooth. The proof relies on interpreting \eqref{eq:PDE} as a geometric heat equation. We also establish basic properties of solutions to \eqref{eq:PDE} that will be useful later in the paper.

We write $u\in C^{i,j}(\R^n\times [a,b])$ to mean that $u$ is continuous in $(x,t)$,  $x\mapsto u(x,t)$ is $i$-times continuously differentiable, and $t\mapsto u(x,t)$ is $j$-times continuously differentiable, on the domain $\R^n\times [a,b]$.

\subsection{Viscosity solution theory}
\label{sec:viscosity}

We recall the definition of viscosity solution of the parabolic PDE
\begin{equation}\label{eq:genPDE}
u_t + F(\nabla^2u,\nabla u) = 0 \ \ \text{ in } \ \ \R^n\times (0,1).
\end{equation}
We let $\usc(\O)$ (resp.~$\lsc(\O)$) denote the set of upper (resp.~lower) semicontinuous functions on a subset $\O$ of Euclidean space. We also denote by $u^*$ and $u_*$ the upper and lower semicontinuous envelopes of $u$, respectively.
\begin{definition}\label{def:viscosity}
We say $u\in \usc(\R^n\times [0,1])$ is a viscosity \emph{subsolution} of \eqref{eq:genPDE} if for all $\phi\in C^\infty(\R^n\times \R)$ and $(x,t)\in \R^n\times (0,1)$ such that $u-\phi$ has a local maximum at $(x,t)$ we have
\begin{equation}\label{eq:viscsub}
\phi_t(x,t) + F(\nabla^2\phi(x,t),\nabla \phi(x,t)) \geq 0.
\end{equation}

Similarly, we say $v\in \lsc(\R^n\times [0,1])$ is a viscosity \emph{supersolution} of \eqref{eq:genPDE} if for all $\phi\in C^\infty(\R^n\times \R)$ and $(x,t)\in \R^n\times (0,1)$ such that $u-\phi$ has a local minimum at $(x,t)$ we have
\begin{equation}\label{eq:viscsuper}
\phi_t(x,t) + F(\nabla^2\phi(x,t),\nabla \phi(x,t)) \leq 0.
\end{equation}

We say $u\in C(\R^n\times [0,1])$ is a \emph{viscosity solution} of \eqref{eq:genPDE} if $u$ is both a viscosity sub-~and supersolution.
\end{definition}
We note that the inequalities in \eqref{eq:viscsub} and \eqref{eq:viscsuper} are flipped, compared to standard definitions in \cite{crandall1992user}, due to the fact that \eqref{eq:genPDE} is a \emph{final-time} value problem. Also, we note that sometimes the superjet and subjet definitions are used in place of the test function definition (see \cite{crandall1992user}). The two definitions are equivalent when $F$ is continuous (see, e.g., \cite{calder2018lecture}). 

Since we work on an unbounded domain, we must restrict the class of super and sub-solutions to those with linear growth.
\begin{definition}\label{def:lineargrowth}
We say $u:\R^n\times [0,1]$ has \emph{linear growth} if there exists $C>0$ such that $|u(x,t)|\leq C(1+|x|)$ for all $(x,t)\in \R^n\times [0,1]$.
\end{definition}

We note that our main equation \eqref{eq:PDE} is discontinuous when $\langle p,\one\rangle = 0$, due to the definition of $Q(\nabla u)$, given in \eqref{eq:Qdef}. 
We work with sub-~and supersolutions that are strictly monotone increasing so as to avoid the discontinuity at $\langle p,\one\rangle =0$.
\begin{definition}\label{def:strictinc}
Let $\theta>0$. We say that $u:\R^n\times [0,1]$ is \emph{$\theta$-increasing} if 
\begin{equation}\label{eq:strictinc}
u(x+s\one,t) \geq u(x,t) + \theta s \ \ \ \text{for all }(x,t)\in \R^n\times [0,1], s\geq 0.
\end{equation}
\end{definition}
Under this definition, \eqref{eq:strict_inc} implies that $g$ is $n\cg$-increasing.

Let $\eps>0$ and consider the modified PDE 
\begin{equation}\label{eq:PDEmod}
u_t + \frac{1}{2^{d+1}}\sum_{\eta\in Q_\eps(\nabla u)} \langle \nabla^2 u\,\eta,\eta\rangle = 0 \ \ \ \text{ in }\R^n\times (0,1),
\end{equation}
where $Q_\eps(\nabla u)$ is defined by
\begin{equation}\label{eq:Qeps}
Q_\eps(p)=\left\{q(m) - \frac{\langle p, q(m)\rangle}{\max\{\langle p, \one\rangle,\eps\}}\one \, : \, m\in \B^d\right\}.
\end{equation}
When $u$ is $\theta$-increasing for $\theta\geq \eps$, solutions of \eqref{eq:PDEmod} and \eqref{eq:PDE} are equivalent, as we show below. It is often more useful to work with the modified equation \eqref{eq:PDEmod}, since \eqref{eq:PDEmod} is continuous in both $\nabla u$ and $\nabla^2u$.

We first record a comparison principle for \eqref{eq:PDE} for linear-growth sub-~and supersolutions that are $\theta$-increasing.
\begin{theorem}\label{thm:comparison}
Assume $g$ is uniformly continuous. Let $u\in \usc(\R^n\times [0,1])$ by a viscosity subsolution of \eqref{eq:PDE} and let $v\in \lsc(\R^n\times [0,1])$ be a viscosity supersolution of \eqref{eq:PDE}. Suppose there exists $C,\theta>0$ such that $u$ and $v$ are $\theta$-increasing and $u(x,t) \leq C(1+|x|)$ and $v(x,t) \geq -C(1+|x|)$ for all $(x,t)\in \R^n\times [0,1]$. Then if 
\[u(x,1)\leq g(x) \leq  v(x,1)\]
for all $x\in \R^n$, then $u\leq v$ on $\R^n\times [0,1]$.
\end{theorem}
\begin{proof}
We claim that for $0<\eps \leq \theta$, $u$ is a viscosity subsolution of \eqref{eq:PDEmod} and $v$ is a viscosity supersolution of \eqref{eq:PDEmod}. Indeed, we will show $u$ is a subsolution; the proof that $v$ is a supersolution is similar. Let $\varphi\in C^\infty(\R^n\times \R)$ and $(x_0,t_0)\in \R^n\times (0,1)$ such that $u-\varphi$ has a local maximum at $(x_0,t_0)$. It follows that
\[u(x,t) - \varphi(x,t) \leq u(x_0,t_0) - \varphi(x_0,t_0)\]
for $(x,t)$ near $(x_0,t_0)$. Setting $t=t_0$ and $x=x_0 + s\one$ for sufficiently small $s$, we have
\[\varphi(x_0+s\one,t_0) - \phi(x_0,t_0) \geq u(x_0+s\one,t_0) - u(x_0,t_0) \geq \theta s,\]
since $u$ is $\theta$-increasing. Dividing by $s$ and sending $s\to 0^+$ we have
\[\langle \nabla \varphi(x_0,t_0), \one \rangle\geq \theta.\]
Hence, if $\eps \leq \theta$, we have $Q_\eps(\nabla \varphi(x_0,t_0)) = Q(\nabla \varphi(x_0,t_0))$, which verifies the subsolution condition.

It is a standard argument (see, e.g., \cite[Section 10.2]{Evans}) that $u$ and $v$ are viscosity sub-~and supersolutions of \eqref{eq:PDEmod} on the extended domain $\R^n\times [0,1)$. Since $u$ and $v$ have at most linear growth, we can apply a standard comparison principle from viscosity solution theory (see, e.g., \cite[Theorem 2.1]{giga1991comparison}) to find that $u\leq v$ on $\R^n\times [0,1]$, which completes the proof.
\end{proof}

We can establish existence of a linear growth viscosity solution with the Perron method.
\begin{theorem}\label{thm:existence}
Assume $g$ is uniformly continuous and $\theta$-increasing for $\theta>0$. Then there exists a viscosity solution $u\in C(\R^n\times [0,1])$ of \eqref{eq:PDE} that has linear growth and is $\theta$-increasing. If $v\in C(\R^n\times [0,1])$ is any other viscosity solution of \eqref{eq:PDE} that has linear growth and is $\bar{\theta}$-increasing for any $\bar{\theta}>0$, then $u=v$.
\end{theorem}
\begin{remark}
From now on, we will refer to the viscosity solution of \eqref{eq:PDE} to mean the unique linear growth $\theta$-increasing viscosity solution.
\label{rem:increasing}
\end{remark}
\begin{proof}
We again work with the modified equation \eqref{eq:PDEmod} with $\eps=\theta$. We will use the Perron method with barrier functions
\[w^\pm_\delta(x,t) := g_\delta(x) \pm K_\delta(1-t) \pm \|g-g_\delta\|_{L^\infty(\R^n)},\]
where $\delta>0$ and $g_\delta:=\eta_\delta * g$, with $\eta_\delta$ a standard mollifier with bandwidth $\delta$. Since $g$ is uniformly continuous, it in fact has linear growth. Thus, the barriers $w^{\pm}_\delta$ are smooth functions with linear growth, and are $\theta$-increasing, since $g$ is $\theta$-increasing. For sufficiently large $K_\delta>0$, $w^+_\delta$ is a classical supersolution of \eqref{eq:PDEmod} and $w^-_\delta$ is a subsolution of \eqref{eq:PDEmod}. We also clearly have
\[w^-_\delta(x,1) \leq g(x) \leq w^+_\delta(x,1) \ \ \ \text{ for all }x\in \R^n.\] 
Since $g$ is uniformly continuous, we have
\[\lim_{\delta\to 0^+}\|g-g_\delta\|_{L^\infty(\R^n)}=0.\]

We now use the Perron method (see, e.g., \cite{crandall1992user} or \cite[Chapter 7]{calder2018lecture}) on the modified equation \eqref{eq:PDEmod}. In particular, we define
\begin{align*}
\F &= \Big\{v \in \usc(\R^n\times [0,1]) \, : \, v \text{ is a linear growth viscosity subsolution of }\eqref{eq:PDEmod}, \\
&\hspace{3in}\text{ and }v(x,1)\leq g(x) \text{ for all }x\in \R^n\Big\},
\end{align*}
and the Perron function $u(x) = \sup\{v(x)\, : \, v\in \F\}$. The set $\F$ is nonempty, since $w^-_\delta\in \F$ for all $\delta>0$. Therefore, it is a standard result that $u^*$ is a viscosity subsolution of \eqref{eq:PDEmod} (see \cite[Lemma 7.1]{calder2018lecture}). By the comparison principle for \eqref{eq:PDEmod}, we have $v \leq w^+_\delta$ for all $\delta>0$. Since $w^-_\delta\in \F$ for all $\delta>0$, we also have $v\geq w^{-}_\delta$.  Therefore $w^{-}_\delta \leq u\leq w^+_\delta$ for all $\delta>0$, and since $w^+_\delta$ is continuous, we have $w^-_\delta \leq u^* \leq w^+_\delta$. In particular, $u^*$ has linear growth and for all $x\in \R^n$
\[u^*(x,1)\leq \lim_{\delta\to 0}w^+_\delta(x,1) = g(x).\]
Therefore $u^*\in \F$ and so $u=u^*$.  It is also a standard result \cite[Lemma 7.2]{calder2018lecture} that $u_*$ is a viscosity supersolution of \eqref{eq:PDEmod}. Since $u\geq w^-_\delta$ we have $u_*\geq w^{-}_\delta$, due to continuity, and so
\[u_*(x,1)\geq \lim_{\delta\to 0}w^-_\delta(x,1) = g(x).\]
Since $w^-_1 \leq u_*\leq u$, we see that $u_*$ has linear growth, and by the comparison principle for \eqref{eq:PDEmod} we have $u\leq u_*$. The opposite inequality is true by definition, and so $u=u^*=u_*$ is the unique linear growth viscosity solution of \eqref{eq:PDEmod} satisfying $u(x,1)=g(x)$.

To see that $u$ is a viscosity solution of \eqref{eq:PDE}, we simply need to show that $u$ is $\theta$-increasing, due to the argument at the start of the proof of Theorem \ref{thm:comparison}. To see this, define $\bar{u}(x,t) = u(x+s\one,t)$. Then $\bar{u}$ is a viscosity solution of \eqref{eq:PDEmod} satisfying $\bar{u}(x,1)=g(x+s\one)$. Since $g$ is $\theta$-increasing we have
\[\bar{u}(x,1) = g(x+s\one) \geq g(x) + \theta s.\]
By the comparison principle for \eqref{eq:PDEmod} we have $\bar{u} \geq u + \theta s$, which establishes that $u$ is $\theta$-increasing.

The uniqueness statement follows from the comparison principle (Theorem \ref{thm:comparison}).
\end{proof}

We now establish some basic properties enjoyed by the solution of \eqref{eq:PDE}.
\begin{proposition}\label{prop:PDEprop}
Assume $g$ is uniformly continuous and $\theta$-increasing. Let $u\in C(\R^n\times [0,1])$ be the viscosity solution of \eqref{eq:PDE}. The following hold.
\begin{enumerate}[(i)]
\item If $g$ is Lipschitz continuous, then for each $t\in [0,1]$ the mapping $x\mapsto u(x,t)$ is Lipschitz continuous. In particular
\[|u(x,t) - u(y,t)| \leq \Lip(g)|x-y|\]
for all $x,y\in \R^n$ and $t\in [0,1]$.
\item If \eqref{eq:strict_inc} holds, then 
\[u(x+v,t) \geq u(x) + \cg \langle v,\one\rangle\]
for all $x\in \R^n, v\in [0,\infty)^n$ and $t\in [0,1]$.
\item If \eqref{eq:translation} holds, then 
\[u(x+s\one,t) = u(x,t) + s\]
for all $s>0$ and $(x,t)\in \R^n\times [0,1]$.
\end{enumerate}
\end{proposition}
\begin{proof}
To prove (i), let $v\in \R^n$ and define
\[w(x,t) = u(x+v,t) + \Lip(g)|v|.\]
Then it is immediate to check that $w$ is a viscosity solution of \eqref{eq:PDE}  satisfying
\[w(x,1) = u(x+v,1) + \Lip(g)|v| = g(x+v) + \Lip(g)|v| \geq g(x).\]
Therefore, by Theorem \ref{thm:comparison} we have $u \leq w$, and so
\[u(x) -u(x+v,t) \leq \Lip(g)|v|.\]
for all $x,v\in \R^n$ and $t\in [0,1]$. Setting $v=y-x$ completes the proof.

To prove (ii), fix $v\in [0,\infty)^n$ and define $w(x,t) = u(x+v,t)- \cg \langle v,\one\rangle$. Then $w$ is a viscosity solution of \eqref{eq:PDE} satisfying
\[w(x,1) = g(x+v)- \cg \langle v,\one\rangle \geq g(x)\]
due to \eqref{eq:strict_inc}. By Theorem \ref{thm:comparison} we have $w(x,t) \geq u(x,t)$, which completes the proof.

The proof of (iii) is similar. We define $w(x,t) = u(x+s\one,t) - s$ and show that $w$ solves the same equation \eqref{eq:PDE}. By uniqueness $w=u$.
\end{proof}

It turns out that the equation \eqref{eq:PDE} is geometric. That is, the equation is unchanged by a relabeling of its level sets. In fact, the level sets evolve according to a linear heat equation, as we show in Section \ref{sec:classical}.
\begin{lemma}\label{lem:geometric}
Let $u\in \usc(\R^n\times [0,1])$ be a $\theta$-increasing viscosity subsolution of \eqref{eq:PDE}. Let $\Psi:\R\to\R$ be smooth with $\Psi'>0$. Then $w(x,t):=\Psi(u(x,t))$ is a viscosity subsolution of \eqref{eq:PDE}.
\end{lemma}
\begin{proof}
Let $\phi\in C^\infty(\R^n\times \R)$ and $(x_0,t_0)\in \R^n\times (0,1)$ such that $w-\phi$ has a local maximum at $(x_0,t_0)$. We may assume $w(x_0,t_0)=\phi(x_0,t_0)$. Then for some $r>0$
\[w(x,t)\leq \phi(x,t) \ \ \ \text{whenever } |x-x_0|\leq r \text{ and }|t-t_0|<r.\]
Define $\psi = \Psi^{-1}(\phi)$. Since $\Psi$ and  $\Psi^{-1}$ are strictly increasing, we have $u(x_0,t_0)=\psi(x_0,t_0)$ and
\[u(x,t)\leq \psi(x,t) \ \ \ \text{whenever } |x-x_0|\leq r \text{ and }|t-t_0|<r.\]
Therefore $u-\psi$ has a local maximum at $(x_0,t_0)$. Since $u$ is $\theta$-increasing, we have
\begin{equation}\label{eq:gradpsi}
\langle \nabla \psi(x_0,t_0),  \one\rangle \geq \theta>0,
\end{equation}
as in the proof of Theorem \ref{thm:comparison}. Thus, by the viscosity subsolution property for $u$ we have
\[\psi_t(x_0,t_0) + \frac{1}{2^{d+1}}\sum_{\eta\in Q(p)} \langle \nabla^2 \psi(x_0,t_0)\,\eta,\eta\rangle \geq 0,\]
where $p=\nabla \psi(x_0,t_0)$. Note we have
\[\varphi_t(x_0,t_0) = \Psi'(u(x_0,t_0))\psi_t(x_0,t_0), \ \ \nabla \varphi(x_0,t_0) = \Psi'(u(x_0,t_0))p,\]
and
\[\nabla^2 \varphi(x_0,t_0) =\Psi'(u(x_0,t_0))\nabla^2 \psi(x_0,t_0) + \Psi''(u(x_0,t_0)) p\otimes p.\] 
Therefore
\[\varphi_t(x_0,t_0) + \frac{1}{2^{d+1}}\sum_{\eta\in Q(p)}\langle \nabla ^2\varphi(x_0,t_0)\eta,\eta\rangle \geq  \frac{\Psi''(u(x_0,t_0))}{2^{d+1}}\sum_{\eta\in Q(p)}|\langle p,\eta\rangle|^2.\]
Since $Q(p)\subset p^\perp$, the right hand side vanishes, so we obtain
\[\varphi_t(x_0,t_0) + \frac{1}{2^{d+1}}\sum_{\eta\in Q(p)}\langle \nabla ^2\varphi(x_0,t_0)\eta,\eta\rangle \geq 0.\]
By \eqref{eq:gradpsi} we have $\langle p,\one\rangle\geq \theta>0$, and so 
\[Q(p) = Q\left( \left( \Psi(u(x_0,t_0)) \right)^{-1}\nabla \varphi(x_0,t_0) \right) = Q(\nabla \varphi(x_0,t_0)),\]
which completes the proof.
\end{proof}

\begin{remark}\label{rem:propsuper}
An analogous statement to Lemma \ref{lem:geometric} holds for supersolutions. That is, if $u\in\lsc(\R^n\times [0,1])$ is a $\theta$-increasing viscosity supersolution of \eqref{eq:PDE} then $w:=\Psi(u)$ is also a viscosity supersolution.
\end{remark}

\subsection{Classical solutions}
\label{sec:classical}

Under some conditions on the payoff $g$ and the expert strategies, the viscosity solution $u$ of \eqref{eq:PDE} has additional regularity and is sometimes a smooth classical solution. This stems from the observation made in Lemma \ref{lem:geometric} that the PDE is geometric. It turns out that, in the right coordinate system, the level sets of the solution $u$ evolve by a linear heat equation that is in some cases uniformly elliptic. 

To see the geometric nature of \eqref{eq:PDE}, we make a change of coordinates as follows: 
\begin{equation}\label{eq:changecoord}
\left\{\begin{aligned}
y_i&=x_i - x_n,&&(1 \leq i \leq n-1)\\
y_n&=x_1 + \cdots + x_n.
\end{aligned}\right.
\end{equation}
That is, we define the matrix
\begin{equation}\label{eq:A}
R = 
\begin{bmatrix}
1&0&0&\cdots &0&-1\\
0&1&0&\cdots &0&-1\\
0&0&1&\cdots &0&-1\\
\vdots&\vdots&\vdots&\ddots&\vdots&\vdots\\
0&0&0&\cdots &1&-1\\
1&1&1&\cdots&1&1\\
\end{bmatrix},
\end{equation}
and make the change of variables $y = Rx$. The inverse coordinate transformation is easily obtained as
\[\left\{\begin{aligned}
x_i&=y_i+\frac{1}{n}y_n - \frac{1}{n}\sum_{i=1}^{n-1} y_i,&&(1 \leq i \leq n-1)\\
x_n&=\frac{1}{n}y_n - \frac{1}{n}\sum_{i=1}^{n-1} y_i.
\end{aligned}\right.\]
In these new coordinates, we now decompose the payoff $g$ into its level-surfaces.
\begin{proposition}\label{prop:levelsets}
Assume $g$ is Lipschitz continuous and $\theta$-increasing. Define $\bar{g}(y) = g(R^{-1}y)$. Then there exists a Lipschitz continuous function $h^0: \R^n \to \R$ such that
\begin{equation}\label{eq:glevel}
\bar{g}(y_1,y_2,\dots,y_{n-1},h^0(y_1,\dots,y_{n-1},s))=s
\end{equation}
holds for all $y\in \R^{n-1}$ and $s\in \R$. Furthermore, the following hold:
\begin{enumerate}[(i)]
\item For all $(y,s)\in \R^{n-1}\times \R$ we have  
\[\sqrt{n}\Lip(g)^{-1} \leq h^0_s(y,s)  \leq n\theta^{-1},\]
\item If \eqref{eq:translation} holds then for all $y\in \R^n$ and $s\in \R$
\[h^0(y_1,\dots,y_{n-1},s) = y_n - n\bar{g}(y) + ns,\]
\item If $g\in C^k(\R^n)$ then $h^0\in C^k(\R^n)$, and $[h^0]_{C^k(\R^n)}$ depends only on $[g]_{C^k(\R^n)}$ and $\theta$.
\end{enumerate}
\end{proposition}
\begin{remark}
The function $y\mapsto h^0(y,s)$ is a parametrization of the level set $\{\bar{g}=s\}$ in the form $y_n = h^0(y,s)$.
\label{rem:levelset}
\end{remark}
\begin{proof}
The proof follows from the implicit function theorem. Notice that \eqref{eq:strict_inc} implies
\begin{equation}\label{eq:gyn}
\bar{g}_{y_n}(y) = \frac{1}{n}\langle \nabla g(R^{-1}y),\one \rangle \geq \frac{1}{n}\theta>0.
\end{equation}
We also have $\bar{g}_{y_n}(y) \leq \frac{1}{\sqrt{n}}\Lip(g)$. It follows that for every $s\in \R$ and $y\in \R^{n-1}$, there is a unique $h^0\in \R$ such that
\[\bar{g}(y_1,y_2,\dots,y_{n-1},h^0)= s. \]
This defines the function $h^0=h^0(y,s)$. Due to \eqref{eq:gyn} the implicit function theorem guarantees that $h^0$ is Lipschitz continuous on $\R^n$.  This establishes the existence of $h^0$.

To prove (i), we differentiate \eqref{eq:glevel} in $s$ to find
\[\bar{g}_{y_n}(y,h^0(y,s)) h^0_s(y,s) = 1,\]
and apply the bounds $\frac{1}{n}\theta \leq \bar{g}_{y_n} \leq \frac{1}{\sqrt{n}}\Lip(g)$ proved above. 

To prove (ii), we note that $h^0$ satisfies
\begin{equation}\label{eq:glevel2}
h^0(y_1,\dots,y_{n-1},\bar{g}(y)) = y_n
\end{equation}
for all $y\in \R^n$, $s\in \R$.  Since \eqref{eq:translation}  holds we have $\langle \nabla g, \one \rangle = 1$, and thus $\bar{g}_{y_n} = \frac{1}{n}$ and $h_s^0=n$. Combining this with \eqref{eq:glevel2} yields
\[y_n = h^0(y_1,\dots,y_{n-1},\bar{g}(y)) = h^0(y_1,\dots,y_{n-1},0) + n\bar{g}(y),\]
Therefore
\[h^0(y_1,\dots,y_{n-1},0) = y_n - n\bar{g}(y).\]
Since $h^0(y,s) = h^0(y,0) + ns$, the claim follows.

The proof of (iii) follows from the implicit function theorem.
\end{proof}

Our first regularity result shows that the level sets $\{u(x,t)=s\}$ evolve by a linear heat equation. When the translation property \eqref{eq:translation} and \eqref{eq:span} hold, this yields a representation formula for the solution of \eqref{eq:PDE}, and we can use the parabolic smoothing from this interpretation to show that $u\in C^\infty$. 
\begin{theorem}\label{thm:reg2}
Assume \eqref{eq:span} and \eqref{eq:translation} hold, and let $g$ be Lipschitz continuous and $\theta$-increasing. Then the viscosity solution $u$ of \eqref{eq:PDE} is given by
\begin{equation}\label{eq:uexplicit}
u(x,t) = h(x_1-x_n,\dots,x_{n-1}-x_n,t)+\frac{1}{n}(x_1+\cdots+x_n),
\end{equation}
where $h\in C^\infty(\R^n\times [0,1))$ is the solution of the heat equation
\begin{equation}\label{eq:hpde2}
\left\{\begin{aligned}
h_{t} + \frac{1}{2^{d+1}}\sum_{m\in \B^d}\langle \nabla^2h r(m),r(m)\rangle &= 0,&&\text{in } \R^{n-1}\times (0,1)\\ 
h(y,1) &=\bar{g}(y,0),&&\text{for } y\in \R^{n-1},
\end{aligned}\right.
\end{equation}
and $\bar{g}(y)=g(R^{-1}y)$. In particular, $u\in C^\infty(\R^n\times [0,1))$ and 
\begin{equation}\label{eq:derivatives}
\left\{\begin{aligned}
|u_{\xi\xi}(x,t)| &\leq \frac{C\Lip(g)}{\sqrt{(1-t)\Cr}}, \ \ &|u_{\xi\xi\xi}(x,t)| \leq \frac{C\Lip(g)}{(1-t)\Cr},\\
|u_t(x,t)| &\leq \frac{C\Lip(g)}{\sqrt{1-t}}, \ \ \text{ and }\ \ &|u_{tt}(x,t)| \leq \frac{C\Lip(g)}{(1-t)^{3/2}},
\end{aligned}\right.
\end{equation}
hold for all $\xi\in \R^n$ with $|\xi|=1$ and all $(x,t)\in \R^n\times [0,1)$.
\end{theorem}
In the theorem statement, we use the notation $u_{\xi\xi}=\langle \nabla^2 u \xi,\xi\rangle$ and $u_{\xi\xi\xi} = \sum_{i,j,k=1}^{n} u_{x_ix_jx_k}\xi_i\xi_j\xi_k$.
We also recall that $r(m)$ is defined in \eqref{eq:rdef}.
\begin{proof}
Let $A$ be defined as follows:
\begin{equation}\label{eq:Acoeff}
A = \frac{1}{2^{d+1}}\sum_{m\in \B^d}r(m)\otimes r(m).
\end{equation}
By \eqref{eq:span} we have $A\geq \Cr I$, so \eqref{eq:hpde2} is uniformly elliptic, and $h\in C^\infty(\R^n\times [0,1))$. We note that \eqref{eq:hpde2} is a nondivergence form equation, which can be written as
\[h_t + \text{Tr}(A\nabla ^2 h)  = 0.\]
Thus, $h$ is given by the solution formula
\begin{equation}\label{eq:HeatForm2}
h(y,t) = \int_{\R^n}\Phi_{A}(y-z,1-t)\bar{g}(z,0) \, dz,
\end{equation}
where $\Phi_A$ is the heat kernel given by
\begin{equation}\label{eq:HeatKernel}
\Phi_A(y,t) = \frac{1}{(4\pi t)^{n/2}|A|^{1/2}}\exp\left( -\frac{\langle A^{-1}y,y\rangle}{4t} \right).
\end{equation}
We can differentiate \eqref{eq:HeatForm2} to obtain the following estimates: There exists $C>0$ such that for all $(y,t) \in \R^n\times [0,1)$ and all $\xi\in \R^{n-1}$ with $|\xi|=1$ 
\begin{equation}\label{eq:estimates}
\left\{\begin{aligned}
|h_{\xi\xi}(y,t)| &\leq \frac{C\Lip(\bar{g})}{\sqrt{(1-t)\Cr}}, \ \ &|h_{\xi\xi\xi}(y,t)| \leq \frac{C\Lip(\bar{g})}{(1-t)\Cr},\\
|h_t(y,t)| &\leq \frac{C\Lip(\bar{g})}{\sqrt{1-t}}, \ \ \text{ and }\ \ &|h_{tt}(y,t)| \leq \frac{C\Lip(\bar{g})}{(1-t)^{3/2}}.\\
\end{aligned}\right.
\end{equation}
 
We now show that $u$ solves \eqref{eq:PDE}. To see this, first note that $u_{x_i} = h_{y_i}+\frac{1}{n}$ for $i=1,\dots,n-1$ and 
\[u_{x_n} =-(h_{y_1}+ \cdots + h_{y_{n-1}}) +  \frac{1}{n}.\]
Therefore
\[\langle \nabla u(x,t),\one\rangle = u_{x_1}+ u_{x_2}+ \cdots + u_{x_n} = 1,\]
and it follows that $\nabla^2 u(x,t)\one =0$ and $\one^T \nabla ^2u(x,t)=0$. Since $q(m) = (r(m),0) + q_n(m)\one$ for all $m\in \B^d$, we thus have
\begin{align*}
\langle \nabla^2 u(x,t) q(m),q(m) \rangle &= \sum_{i,j=1}^{n-1} u_{x_ix_j}(x,t)r_i(m)r_j(m)\\
&=\sum_{i,j=1}^{n-1}h_{y_iy_j}(x_1-x_n,\dots,x_{n-1}-x_n,t) r_i(m)r_j(m)\\
&=\langle \nabla^2 hr(m),r(m)\rangle.
\end{align*}
Since $u_t = h_t$ we find that $u$ satisfies
\begin{equation}\label{eq:otherHeat}
u_t + \frac{1}{2^{d+1}}\sum_{m\in \B^d}\langle \nabla^2u\, q(m),q(m)\rangle = 0.
\end{equation}
Each $\eta \in Q(\nabla u)$ is of the form
\[\eta = q(m) - \frac{\langle \nabla u, q(m)\rangle}{\langle \nabla u, \one\rangle}\one = q(m) - \langle\nabla u, q(m)\rangle\one,\]
for $m\in \B^d$, and so
\[\sum_{\eta \in Q(\nabla u)}\langle \nabla^2 u\, \eta,\eta\rangle = \sum_{m\in \B^d}\langle \nabla^2u \,q(m),q(m)\rangle.\]
Therefore
\[u_t + \frac{1}{2^{d+1}}\sum_{\eta \in Q(\nabla u)}\langle \nabla^2u\, \eta,\eta\rangle = 0.\]
Finally, we check the final condition $u(x,1)=g(x)$. As in the proof of Proposition \ref{prop:levelsets} (ii) we have $\bar{g}_{y_n}=\frac{1}{n}$, and so
\begin{align*}
u(x,1) &= \bar{g}(x_1-x_n,\dots,x_{n-1}-x_n,0)+\frac{1}{n}(x_1+\cdots + x_n) \\
&= \bar{g}\left(x_1-x_n,\dots,x_{n-1}-x_n,\tfrac{1}{n}(x_1+\cdots + x_n)\right) = g(x),
\end{align*}
which completes the proof.
\end{proof}
\begin{remark}
Notice that in the proof of Theorem \ref{thm:reg2}, we showed that $u$ solves the linear heat equation \eqref{eq:otherHeat}. This depends crucially on the translation property \eqref{eq:translation} holding. In this case, we can replace \eqref{eq:span} with the condition that
\begin{equation}\label{eq:otherspan}
B:=\frac{1}{2^{d+1}}\sum_{m\in \B^d}q(m)\otimes q(m) \geq \theta I
\end{equation}
for some $\theta>0$, and the results of Theorem \ref{thm:reg2} continue to hold. However, we claim that \eqref{eq:otherspan} implies \eqref{eq:span}, and so the condition \eqref{eq:span} is more general. To see this, assume \eqref{eq:otherspan} holds, and note that 
\[q(m) = (r(m),0) + q_n(m)\one.\]
Let $\xi\in \R^{n-1}$ and choose $\xi_n=-(\xi_1+\cdots+\xi_{n-1})$ so that $\langle \one, (\xi,\xi_n)\rangle=0$. Then
\[\langle q(m),(\xi,\xi_n)\rangle_{\R^n} = \langle r(m),\xi\rangle_{\R^{n-1}}.\]
Therefore, for $A$ given by \eqref{eq:Acoeff} we have
\begin{align*}
\langle A\xi,\xi\rangle_{\R^{n-1}} &= \frac{1}{2^{d+1}}\sum_{m\in \B^d}|\langle r(m),\xi\rangle_{\R^{n-1}}|^2\\
&=\frac{1}{2^{d+1}}\sum_{m\in \B^d}|\langle q(m),(\xi,\xi_n)\rangle_{\R^n}|^2\\
&=\langle B(\xi,\xi_n),(\xi,\xi_n)\rangle_{\R^{n-1}} \geq \theta (|\xi|^2 + |\xi_n|^2) \geq \theta|\xi|^2,
\end{align*}
which establishes the claim.
\label{rem:otherspan}
\end{remark}

When the translation property \eqref{eq:translation} does not hold, the situation is more complicated. Following similar ideas to Theorem \ref{thm:reg2}, we show below that the level sets $\{u(x,t)=s\}$ evolve by the same heat equation. However, we loose the parabolic smoothing across level sets in this case, and  thus we require additional regularity for $g$.
\begin{theorem}\label{thm:reg1}
Assume $g\in C^4(\R^n)$, $g$ is $\theta$-increasing, and let $u\in C(\R^n\times [0,1])$ be the viscosity solution of \eqref{eq:PDE}. Then, $u\in C^{4,2}(\R^n\times [0,1))$ with $[u(\cdot,t)]_{C^{4}(\R^n)}$ and $[u(x,\cdot)]_{C^2([0,1])}$ depending only on $\theta$ and $[g]_{C^4(\R^n)}$ for all $(x,t)\in \R^n\times [0,1)$.
\end{theorem}
\begin{proof}
The proof is split into several steps.

1. For $\eps>0$, define the function $h_\eps:\R^{n-1}\times [0,1]\times \R\to \R$ so that for every $s\in \R$ the function  $(y,t)\mapsto h_\eps(y,t,s)$ is the solution of the linear heat equation
\begin{equation}\label{eq:hpde}
\left\{\begin{aligned}
h_{\eps,t} + \frac{1}{2^{d+1}}\sum_{m\in \B^d}\langle \nabla^2h_\eps r(m),r(m)\rangle +\eps\Delta h_\eps&= 0,&&\text{in } \R^{n-1}\times (0,1)\\ 
h_\eps(y,1,s) &=h^0(y,s),&&\text{for } y\in \R^{n-1},
\end{aligned}\right.
\end{equation}
where $h^0$ is defined in Proposition \ref{prop:levelsets}. We will often drop the dependence on $\eps$  for notational convenience. As in the proof of Theorem \ref{thm:reg2}, the solution of \eqref{eq:hpde} is given by
\begin{equation}\label{eq:HeatForm}
h(y,t,s) = \int_{\R^n}\Phi_{A+\eps I}(y-z,1-t)h^0(z,s) \, dz.
\end{equation}
By Proposition \ref{prop:levelsets}, $h^0\in C^4(\R^n)$, and so $h\in C^4(\R^n\times [0,1))$. Furthermore, we can differentiate formula \eqref{eq:HeatForm} to obtain for all  $(y,t,s)\in \R^{n-1}\times [0,1]\times \R$ the following estimates, independent of $\eps>0$:
\begin{equation}\label{eq:derest}
\left\{\begin{aligned}
|D_{(y,s)}^\alpha h(y,t,s)| &\leq \|D^\alpha h^0\|_{L^\infty(\R^n)}, \ \ \ 1\leq |\alpha|\leq 4\\
|h_t(y,t,s)| &\leq C\|D^2_y h^0\|_{L^\infty(\R^n)},\\
|h_{tt}(y,t,s)| &\leq C\|D^4_y h^0\|_{L^\infty(\R^n)}.
\end{aligned}\right.
\end{equation}
We can also differentiate \eqref{eq:HeatForm} in $s$ and apply Proposition \ref{prop:levelsets} to obtain
\begin{equation}\label{eq:forIFT}
\sqrt{n}\Lip(g)^{-1} \leq h_s(y,t,s)  \leq n\theta^{-1},
\end{equation}
for all $(y,t)\in R^{n-1}\times [0,1]$ and $s\in \R$.

2. By \eqref{eq:forIFT}, for every $\eps>0$ and $(y,t)\in R^{n}\times [0,1]$ there exists a unique $v_\eps\in \R$ such that $h_\eps(y_1,\dots,y_{n-1},t,v_\eps) = y_n$. This defines a function $v_\eps:\R^n\times [0,1]\to \R$ that satisfies
\begin{equation}\label{eq:vIFT}
h_\eps(y_1,\dots,y_{n-1},t,v_\eps(y,t)) = y_n
\end{equation}
for all $(y,t)\in \R^n\times [0,1]$. We again drop the subscript $\eps$ for convenience. By \eqref{eq:forIFT} and the implicit function theorem, $v\in C^4(\R^n\times [0,1))$. We can differentiate \eqref{eq:vIFT} and use \eqref{eq:derest} and Proposition \ref{prop:levelsets} (iii) to find that $[v(\cdot,t)]_{C^4(\R^n)}$ and $[v(y,\cdot)]_{C^2([0,1])}$ are bounded depending only on $\theta$ and $[g]_{C^4(\R^n)}$, and are in particular independent of $\eps>0$. We also compute
\[h_s(y_1,\dots,y_{n-1},t,v(y,t))v_{y_n}(y,t) = 1,\]
from which we obtain 
\begin{equation}\label{eq:vynbound}
0< \frac{1}{n}\theta \leq v_{y_n}(y,t)  \leq \frac{1}{\sqrt{n}}\Lip(g).
\end{equation}
Finally, we note that $v$ also satisfies
\begin{equation} \label{eq:vlevel}
v(y_1, \dots, y_{n-1}, h(y_1, \dots, y_{n-1}, t, s), t) = s
\end{equation}
for all $(y,t)\in \R^n\times [0,1]$ and $s\in \R$. 

3. We now derive a PDE satisfied by $v$. Differentiating \eqref{eq:vlevel} in $y_i$ for $1\leq i \leq n-1$ we have
\begin{equation}\label{eq:vyi}
v_{y_i} + v_{y_{n}}h_{y_i} =0,
\end{equation}
and differentiating in $t$ yields
\begin{equation}\label{eq:vt}
v_t + v_{y_n}h_t = 0.
\end{equation}
Note that in all formulas, we evaluate at $(y,t)\in \R^n\times [0,1)$ and set $s=v(y,t)$. Differentiating \eqref{eq:vyi} in $y_j$ yields
\[v_{y_iy_j} + v_{y_iy_n}h_{y_j} + (v_{y_jy_n}+v_{y_ny_n}h_{y_j})h_{y_i} + v_{y_{n}}h_{y_iy_j}  =0.\]
Multiply by $v_{y_n}^2$ on both sides and use \eqref{eq:vyi} to obtain
\[v_{y_n}^2v_{y_iy_j} - v_{y_n}v_{y_iy_n}v_{y_j} - (v_{y_n}v_{y_jy_n}-v_{y_ny_n}v_{y_j})v_{y_i} + v_{y_{n}}^3h_{y_iy_j}  =0,\]
which simplifies to
\begin{equation}\label{eq:id}
v_{y_n}^2v_{y_iy_j} - v_{y_n}(v_{y_iy_n}v_{y_j} + v_{y_jy_n}v_{y_i}) + v_{y_ny_n}v_{y_i}v_{y_j} + v_{y_{n}}^3h_{y_iy_j}  =0.
\end{equation}
Let $\xi \in \R^n$ with $\langle\xi,\nabla v\rangle = 0$. This implies that
\[\sum_{i=1}^{n-1}\xi_iv_{y_i} = -\xi_nv_{y_n}.\]
Multiply by $\xi_i\xi_j$ on both sides of \eqref{eq:id}, sum over $1\leq i,j\leq n-1$, and use the identity above to obtain 
\[v_{y_n}^2 \sum_{i,j=1}^{n-1}v_{y_iy_j}\xi_i\xi_j + 2v_{y_n}^2 \sum_{i=1}^{n-1}v_{y_iy_n}\xi_i\xi_n + v_{y_n}^2 v_{y_ny_n}\xi_n^2 + v_{y_n}^3 \sum_{i,j=1}^{n-1}h_{y_iy_j}\xi_i\xi_j= 0.\]
It follows that whenever $\langle\xi, \nabla v\rangle =0$ we have
\begin{equation}\label{eq:vvv}
\langle \nabla^2v \,\xi,\xi\rangle =-v_{y_n}\sum_{i,j=1}^{n-1}h_{y_iy_j}\xi_i\xi_j.
\end{equation}

We define  $\bar{Q}(p)\subset p^\perp$ by 
\begin{equation}\label{eq:Rdef}
\bar{Q}(p)=\left\{Rq(m) - \frac{\langle p,Rq(m)\rangle}{p_n}e_n \, : \, m\in \B^d\right\}.
\end{equation}
Let $\xi\in \bar{Q}(\nabla v)$ and $m\in \B$ such that
\begin{equation}\label{eq:xiform}
\xi = Rq(m) - \frac{\langle p, Rq(m)\rangle}{p_n}e_n.
\end{equation}
By the definition of $R$ we have
\[\xi_i = q_i(m) - q_n(m) = r_i(m)\]
for all $i\leq n-1$. Since $\langle\xi,\nabla v \rangle=0$, we have by \eqref{eq:vvv} that
\begin{equation}\label{eq:eta_r}
\langle \nabla^2v \,\xi,\xi\rangle= -v_{y_n}\sum_{i,j=1}^{n-1}h_{y_iy_j}r_i(m)r_j(m)=-v_{y_n}\langle \nabla^2h\, r(m),r(m)\rangle.
\end{equation}
We multiply \eqref{eq:hpde} by $-v_{y_n}$ and use $v_t= - v_{y_n}h_t$ to obtain
\[v_t -v_{y_n}\frac{1}{2^{d+1}}\sum_{m\in \B^d}\langle \nabla^2h\, r(m),r(m)\rangle  - \eps v_{y_n}\Delta h= 0.\]
We substitute \eqref{eq:eta_r} in the above and use that $v_{y_n}|\Delta h|\leq C$, with $C$ independent of $\eps$, to obtain
\begin{equation}\label{eq:vpde}
v_t + \frac{1}{2^{d+1}}\sum_{\xi\in \bar{Q}(\nabla v)}\langle \nabla^2v\, \xi,\xi\rangle = O(\eps).
\end{equation}
To check the final time condition, we note that by \eqref{eq:vIFT} evaluated at $t=1$ we have
\[h^0(y_1,\dots,y_{n-1},v(y,1)) = y_n.\]
Comparing this with \eqref{eq:glevel2} in the proof of Proposition \ref{prop:levelsets}, we see that $v(y,1) = \bar{g}(y)=g(R^{-1}y)$.

4. Define $u_\eps(x,t)=v_\eps(Rx,t)$, and compute 
\begin{equation}\label{eq:chainrule}
\nabla_x u_\eps(x,t) = R^T \nabla_y v_\eps(y,t)\ \ \  \text{ and }\ \ \ \nabla_x^2 u_\eps(x,t) = R^T \nabla^2_y v_\eps(y,t)R.
\end{equation}
If follows that $\langle \nabla_x^2 u_\eps\, \eta,\eta\rangle = \langle \nabla_y^2 v_\eps R\eta,R\eta\rangle$. Set $\xi = R\eta$ so that
\[\langle \nabla_y^2 v_\eps \xi,\xi\rangle=\langle \nabla_x^2 u_\eps\, R^{-1}\xi,R^{-1}\xi\rangle.\]
where $\xi\in \bar{Q}(\nabla v)$. We claim that $R^{-1}\bar{Q}(p) = Q(R^Tp)$, where $Q$ is given in \eqref{eq:Qdef}. To see this, let $\xi \in \bar{Q}(p)$ and $m\in \B^d$ so that \eqref{eq:xiform} holds. We write
\begin{align*}
R^{-1}\xi&= q(m) - \frac{\langle p,Rq(m)\rangle }{\langle p,e_n\rangle}R^{-1}e_n\\
&= q(m) - \frac{\langle R^Tp,q(m)\rangle }{\langle R^Tp,R^{-1}e_n\rangle}R^{-1}e_n.
\end{align*}
Since $R^{-1}e_n = \frac{1}{n}\one$, this establishes the claim. Therefore $R^{-1}\bar{Q}(\nabla_y v_\eps) = Q(R^T \nabla_x u_\eps)$ and we have
\[\sum_{\xi\in \bar{Q}(\nabla_y v_\eps)}\langle \nabla_y^2 v_\eps \xi,\xi\rangle=\sum_{\eta\in Q(\nabla_x u_\eps)}\langle \nabla_x^2 u_\eps\, \eta,\eta\rangle.\]
Combining this with \eqref{eq:vpde} we have
\begin{equation}
u_{\eps,t} + \frac{1}{2^{d+1}}\sum_{\eta \in Q(\nabla u_\eps)} \langle \nabla^2 u_\eps\,\eta,\eta \rangle = O(\eps).
\end{equation}
Since $v(y,1) = g(R^{-1}y)$ we have the final time condition $u(x,1) = g(x)$. By \eqref{eq:vynbound} we have
\[\langle \nabla u_{\eps},\one\rangle = nv_{y_n} \geq \theta.\]
Therefore $u_\eps$ is $\theta$-increasing. Sending $\eps\to 0$ we find that $u_\eps\to u$, where $u$ is the viscosity solution of \eqref{eq:PDE}, which completes the proof. 
\end{proof}

\section{Convergence proofs}
\label{sec:proofs}

We now give the proofs of our main convergence results. The proofs rely on a common lemma.
\begin{lemma}\label{lem:scheme}
Assume \eqref{eq:redund} holds. Let $N\geq 1$, $k\geq d+1$, and set $\eps = N^{-1/2}$. Let $(x_0,t_0) \in \R^n \times [0,1]$ and let $\phi\in C^{3,2}(\R^n\times [0,t_0])$. Assume there exists $\gamma>0$ such that $\varphi_{x_i}(x_0,t_0) \geq \gamma$ for all $i\in \{1,\dots,n\}$, and  set
\[c_t = \sup_{t \in [0,t_0]}|\phi_{tt}(x_0,t)|, \ \ c_{x,2} = \sup_{\substack{x\in \R^n\\ |\xi|=1}}|\phi_{\xi\xi}(x,t_0)|, \ \ \text{ and } \ \ c_{x,3} = \sup_{\substack{x\in \R^n\\ |\xi|=1}}|\phi_{\xi\xi\xi}(x,t_0)|.\]
There exists $c>0$, depending only on $n$, such that when $c_{x,2}k\eps \leq c \,\cq \gamma$ and $t_0 - k\eps^2 \geq 0$ the following hold.
\begin{enumerate}[(i)]
\item  If
\[\phi_t(x_0,t_0) + \frac{1}{2^{d+1}}\sum_{\eta \in Q(\nabla \phi(x_0,t_0))}\langle \nabla^2 \phi(x_0,t_0)  \eta,\eta\rangle \leq 0\]
then 
\begin{align*}
u^+_N(x_0,t_0-k\eps^2) - \phi(x_0,t_0-k\eps^2)&\leq \sup_{x\in B(x_0,2k\eps\sqrt{n})}(u^+_N(x,t_0) - \phi(x,t_0))\\
&\hspace{1in} + C\left( c_{x,2}d\eps^2 + c_{x,2}^2\gamma^{-1}k^2\eps^3+  c_{x,3}k^3\eps^3 +c_t k^2\eps^4\right).\notag
\end{align*}
\item  If
\[\phi_t(x_0,t_0) + \frac{1}{2^{d+1}}\sum_{\eta \in Q(\nabla \phi(x_0,t_0))}\langle \nabla^2\phi(x_0,t_0) \eta,\eta\rangle \geq 0\]
then 
\begin{align*}
u^-_N(x_0,t_0-k\eps^2) - \phi(x_0,t_0-k\eps^2)&\geq \inf_{x\in B(x_0,2k\eps\sqrt{n})}(u^-_N(x,t_0) - \phi(x,t_0))\\
&\hspace{1in} - C\left( c_{x,2}d\eps^2 + c_{x,2}^2\gamma^{-1}k^2\eps^3+  c_{x,3}k^3\eps^3 +c_t k^2\eps^4\right).\notag
\end{align*}
\end{enumerate}
\end{lemma}
\begin{proof}
We will prove (i); the proof of (ii) is similar. Let us write
\[M = \sup_{x\in B(x_0,2k\eps\sqrt{n})}(u^+_N(x,t_0) - \phi(x,t_0)).\]
Let $m\in \B^d$ such that 
\[u_N^+(x,t_0-k\eps^2) = u_{N}(x_0,t_0-k\eps^2;m).\]
Then by Proposition \ref{prop:dppuN} we have
\begin{align*}
u^+_{N}(x_0,t_0-k\eps^2) &= u_{N}(x_0,t_0-k\eps^2;m)\\
&=\min_{|f_1|\leq 1}\max_{b_1=\pm 1}\cdots \min_{|f_k|\leq 1}\max_{b_k=\pm 1} \left\{u_N\left( x_0+\eps \sum_{i=1}^k b_i \delta_i, t_0;m^{k+1}\right)\right\}\\
 &\leq \min_{|f_1|\leq 1}\max_{b_1=\pm 1}\cdots \min_{|f_k|\leq 1}\max_{b_k=\pm 1} \left\{u_N^+\left( x_0+\eps \sum_{i=1}^k b_i \delta_i, t_0\right)\right\}\\
 &\leq \min_{|f_1|\leq 1}\max_{b_1=\pm 1}\cdots \min_{|f_k|\leq 1}\max_{b_k=\pm 1} \left\{\phi\left( x_0+\eps \sum_{i=1}^k b_i \delta_i, t_0\right)\right\} + M,
\end{align*}
where $m^1=m$, $m^{i+1}=m^i|b_i$, and $\delta_i = q(m^i) - \one f_i$. Taylor expanding $\phi$ we have
\begin{align*}
\phi\left( x_0+\eps \sum_{i=1}^k b_i \delta_i, t_0\right) &= \phi(x_0,t_0) + \eps\sum_{i=1}^k b_i\langle \nabla  \phi ,\delta_i\rangle + \frac{\eps^2}{2}\sum_{i,j=1}^k b_ib_j \langle \nabla^2\phi\, \delta_i,\delta_j\rangle + O(c_{x,3} k^3\eps^3),
\end{align*}
where $\nabla \varphi$ and $\nabla^2\varphi$ are evaluated at $(x_0,t_0)$. We also have
\[\phi(x_0,t_0) =\phi(x_0,t_0-k\eps^2)   + k\eps^2 \phi_t(x_0,t_0) + O(c_tk^2\eps^4).\]
Plugging this in above and invoking Theorem \ref{thm:cell} we obtain
\begin{align*}
&u_N^+(x_0,t_0-k\eps^2) - \phi(x_0,t_0-k\eps^2)\\
&\hspace{0.5in}\leq k\eps^2 \phi_t(x_0,t_0) + \min_{|f_1|\leq 1}\max_{b_1=\pm 1}\cdots \min_{|f_k|\leq 1}\max_{b_k=\pm 1} \left\{\eps\sum_{i=1}^k b_i\langle \nabla \phi ,\delta_i\rangle + \frac{\eps^2}{2}\sum_{i,j=1}^k b_ib_j \langle \nabla^2\phi\,\delta_i,\delta_j\rangle\right\} \\
&\hspace{4.4in}+ C(c_{x,3} k^3\eps^3 + c_t k^2\eps^4)+ M\\
&\hspace{0.5in}=k\eps^2 \phi_t(x_0,t_0) + \eps^2 \L_{k,\eps}(\nabla^2 \phi(x_0,t_0),\nabla \phi(x_0,t_0),m) +C(c_{x,3} k^3\eps^3 + c_t k^2\eps^4)+ M\\
&\hspace{0.5in}\leq k\eps^2\left( \phi_t + \frac{1}{2^{d+1}}\sum_{\eta \in Q(\nabla u)}\langle \nabla^2 \phi\, \eta,\eta\rangle \right) +C\left( c_{x,2}d\eps^2 + c_{x,2}^2\gamma^{-1}k^2\eps^3+  c_{x,3}k^3\eps^3 +c_t k^2\eps^4\right)+ M\\
&\hspace{0.5in}=C\left( c_{x,2}d\eps^2 + c_{x,2}^2\gamma^{-1}k^2\eps^3+  c_{x,3}k^3\eps^3 +c_t k^2\eps^4\right)+ M,
\end{align*}
provided $c_{x,2}k\eps \leq c \,\cq \gamma$, which completes the proof.
\end{proof}

We now give the proof of Theorem \ref{thm:main1}.
\begin{proof}[Proof of Theorem \ref{thm:main1}]
We first assume $g\in C^4(\R^n)$ with $[g]_{C^4(\R^n)}< \infty$. Since \eqref{eq:strict_inc} holds, we can apply Theorem \ref{thm:reg1} to show that $u\in C^{4,2}(\R^n\times [0,1])$ and the constants $c_t$, $c_{x,2}$ and $c_{x,3}$ from Lemma \ref{lem:scheme} are uniformly bounded depending only on $[g]_{C^4(\R^n)}$ and $\cg$. We continue to denote these constants for completeness, using the definitions
\[c_t = \sup_{\R^n\times [0,1]}|u_{tt}(x,t)|, \ \ c_{x,2} = \sup_{\R^n\times [0,1]}\max_{|\xi|=1}|u_{\xi\xi}|, \ \ \text{ and } \ \ c_{x,3} = \sup_{\R^n\times [0,1]}\max_{|\xi|=1}|u_{\xi\xi\xi}|.\]

Set $\eps=N^{-1/2}$ for convenience. By Definition \eqref{def:value} and equation \eqref{eq:uN}, for any $0 \leq j \leq N$ we have
\begin{align*}
u_N(x,1-j\eps^2;m)&= \eps V_N(\eps^{-1}x,\lceil N(1-j\eps^2)\rceil;m)\\
&=\eps V_N(\eps^{-1}x,N-j;m)\\
&=\eps\min_{|f_{N-j}|\leq 1}\max_{b_{N-j}=\pm 1}\cdots \min_{|f_{N-1}|\leq 1}\max_{b_{N-1}=\pm 1}g\left( \eps^{-1}x + \sum_{i=N-j}^{N-1} b_i(q(m^i) - f_i\one) \right)\\
&\leq \eps (g(\eps^{-1}x) + C \Lip(g) j)= g(x) + C \Lip(g)j\eps,
\end{align*}
due to \eqref{eq:homogeneous}. Therefore, for $0 \leq j \leq k$ we have
\[u_N^+(x,1-j\eps^2) -g(x) \leq C\Lip(g) k\eps.\]
Since $|u_t| \leq C c_{x,2}$, with $C$ depending only on $n$, we have
\[g(x) - u(x,1-t) = u(x,1) - u(x,1-t) \leq C c_{x,2}t.\]
Therefore, for $0 \leq j \leq k$ we have
\begin{align}\label{eq:ugbase}
u_N^+(x,1-j\eps^2) -u(x,1-j\eps^2)  &= u_N^+(x,1-j\eps^2) - g(x) + g(x) - u(x,1-j\eps^2)\\
&\leq C(\Lip(g) k\eps +  c_{x,2} k\eps^2).\notag
\end{align}

Since \eqref{eq:strict_inc} holds,  Proposition \ref{prop:PDEprop} (ii) yields $u_{x_i}\geq \cg$ for all $i\in \{1,\dots,n\}$.  Thus, we can apply Lemma \ref{lem:scheme} with $\varphi=u$ and $\gamma=\cg$ to find that
\begin{align}\label{eq:key}
\sup_{x\in \R^n}(u^+_N(x,t-k\eps^2) - u(x,t-k\eps^2))&\leq \sup_{x\in \R^n}(u^+_N(x,t) - u(x,t))\\
&\hspace{0.5in} + C\left( c_{x,2}d\eps^2 + c_{x,2}^2\cg^{-1}k^2\eps^3+  c_{x,3}k^3\eps^3 +c_t k^2\eps^4\right)\notag
\end{align}
for all $0\leq t \leq 1$ and $d+1\leq k \leq Cc_{x,2}^{-1}\cq \cg \eps^{-1}$ for which $t - k\eps^2 \geq 0$. We recall $\cq$ is defined in \eqref{eq:thetaq_def}. Now fix $0 \leq j \leq k-1$ and $\ell \in \N$ and apply \eqref{eq:ugbase} and then \eqref{eq:key} $\ell$ times to obtain
\begin{align*}
&\sup_{x\in \R^n}(u^+_N(x,1-(j+\ell k)\eps^2) - u(x,1-(j+ \ell k)\eps^2))\leq C \ell \left( c_{x,2}d\eps^2 + c_{x,2}^2\cg^{-1}k^2\eps^3+  c_{x,3}k^3\eps^3 +c_t k^2\eps^4\right)\\
&\hspace{4.75in} + C(\Lip(g) k\eps + c_{x,2} k\eps^2),
\end{align*}
provided $t-(j+\ell k)\eps^2\leq 1$. For every $t\in [0,1]$, $\lceil t\rceil =1- (j+\ell k)\eps^2$ for some $0 \leq j \leq k-1$ and $\ell\in \N$. Hence, we obtain for any $t\in [0,1]$ that
\[\sup_{x\in \R^n}(u^+_N(x,t) - u(x,t))\leq C(1-t) \left( \frac{d}{k}c_{x,2} + c_{x,2}^2\cg^{-1}k\eps+  c_{x,3} k^2\eps +c_t k\eps^2\right) + C(\Lip(g) k\eps + c_{x,2} k\eps^2).\]
Optimizing over $k$ yields $k=\lceil d^{1/3}\eps^{-1/3}\rceil$, and so
\begin{align*}
&\sup_{x\in \R^n}(u^+_N(x,t) - u(x,t))\leq C(1-t)d^{2/3}\eps^{1/3} \left( c_{x,2}+ c_{x,3}+ c_{x,2}^2\cg^{-1}\eps^{1/3}+c_t \eps^{4/3}\right) \\
&\hspace{4in}+ Cd^{1/3}\eps^{2/3}(\Lip(g) + c_{x,2} \eps),
\end{align*}
provided $d+1 \leq d^{1/3}\eps^{-1/3} \leq Cc_{x,2}^{-1}\cq \cg \eps^{-1}$. This is equivalent to $\eps \leq d/(d+1)^3$ and $\eps^{2/3}\leq Cd^{-1/3}c_{x,2}^{-1}\cq \cg$; in other words
\[\eps\leq \min\left\{ \frac{d}{(d+1)^3},Cd^{-1/2}(c_{x,2}^{-1}\cg\cq)^{3/2} \right\},\]
which is equivalent to \eqref{eq:Ncond}, after allowing $C$ to depend on $c_{x,2}$ and recalling $\eps=N^{-1/2}$.

A similar argument shows that
\begin{align*}
&\inf_{x\in \R^n}(u^-_N(x,t) - u(x,t))\geq -C(1-t)d^{2/3}\eps^{1/3} \left( c_{x,2}+ c_{x,3}+ c_{x,2}^2\cg^{-1}\eps^{1/3}+c_t \eps^{4/3}\right) \\
&\hspace{4in}- Cd^{1/3}\eps^{2/3}(\Lip(g) + c_{x,2} \eps).
\end{align*}
This completes the proof in the case that $g\in C^4(\R^n)$, upon allowing $C$ to depend on $[g]_{C^4(\R^n)}$ and $\cg$.

If $g$ is uniformly continuous, then we let $\delta>0$ and define $g^\delta = \eta_\delta * g$, where $\eta_\delta$ is a standard mollifier with bandwidth $\delta>0$. By the uniform continuity of $g$, $g^\delta \to g$ uniformly on $\R^n$ as $\delta\to 0^+$. We define
\[u^\delta_{N}(x,t;m) = \min_{|f_{\lceil Nt\rceil}|\leq 1}\max_{b_{\lceil Nt\rceil}=\pm 1}\cdots \min_{|f_{N-1}|\leq 1}\max_{b_{N-1}=\pm 1} g^\delta\left( x + N^{-1/2}\sum_{i=\lceil Nt\rceil}^{N-1} b_i(q(m^i) - f_i\one) \right),\]
where $m^1=m$ and $m^{i+1}=m^i|b_i$. Since $u_N = u^0_N$ we have 
\begin{equation}\label{eq:discrete}
|u_N-u^\delta_N| \leq \|g - g^\delta\|_{L^\infty(\R^n)}.
\end{equation}
Since $g^\delta\in C^\infty(\R^n)$ and \eqref{eq:strict_inc} holds, the argument above yields that $u^\delta_{N} \to u^\delta$ uniformly on $\R^n\times [0,1]$ as $N\to \infty$, where $u^\delta$ is the viscosity solution of
\begin{equation}\label{eq:PDEdelta}
\left\{\begin{aligned}
u^\delta_t + \frac{1}{2^{d+1}}\sum_{\eta\in Q(\nabla u^\delta)} \langle \nabla^2 u^\delta\,\eta,\eta\rangle &= 0,&&\text{in }\R^n\times (0,1)\\ 
u^\delta &=g^\delta,&&\text{on }\R^n\times \{t=1\}.
\end{aligned}\right.
\end{equation}
By the comparison principle (Theorem \ref{thm:comparison}), we have 
\[\|u - u^\delta\|_{L^\infty(\R^n\times [0,1])}\leq \|g - g^\delta\|_{L^\infty(\R^n)}.\]
Combining this with \eqref{eq:discrete} and the triangle inequality we have
\[\|u^\pm_N - u\|_{L^\infty(\R^n\times [0,1])} \leq 2 \|g - g^\delta\|_{L^\infty(\R^N)} + \max_{m\in \B^d}\|u_N^\delta(\cdot,\cdot;m) - u^\delta\|_{L^\infty(\R^n\times [0,1])}.\]
In particular,
\[\limsup_{N\to \infty}\|u_N - u\|_{L^\infty(\R^n\times [0,1])}  \leq 2 \|g - g^\delta\|_{L^\infty(\R^N)}\]
for all $\delta>0$. Sending $\delta\to 0$ completes the proof. 
\end{proof}

We now give the proof of Theorem \ref{thm:main2}.
\begin{proof}[Proof of Theorem \ref{thm:main2}]
By Theorem \ref{thm:reg2}, $u\in C^\infty(\R^n\times [0,1))$.   As in the proof of Theorem \ref{thm:main1} we have
\[u_N^+(x,1-j\eps^2) -g(x) \leq C\Lip(g) j\eps\]
for all $0 \leq j \leq N$. Due to \eqref{eq:derivatives} from Theorem \ref{thm:reg2}, we have
\begin{align*}
g(x) - u(x,1-t) &=u(x,1) - u(x,1-t)\\
&=\int_{1-t}^1 u_t(x,s) \, ds\\
&\leq C\Lip(g)\int_{1-t}^1 \frac{1}{\sqrt{1-s}}\, ds  = 2C\Lip(g) \sqrt{t},
\end{align*}
for all $t>0$. Let $M\geq 1$, to be determined later. Then for all $0 \leq j \leq 2M$ we have
\begin{align}\label{eq:ugbase2}
u_N^+(x,1-j\eps^2) -u(x,1-j\eps^2) &= u_N^+(x,1-j\eps^2) - g(x) + g(x) - u(x,1-j\eps^2)\\
&\leq C\Lip(g) j\eps + 2C\Lip(g) \sqrt{j}\eps \leq C\Lip(g)M\eps.\notag
\end{align}

As in the proof of Theorem \ref{thm:main1}, we now apply Lemma \ref{lem:scheme} with $\varphi=u$ and $\gamma=\cg$ for $t \leq 1-M \eps^2$. Due to Theorem \ref{thm:reg2}, Lemma \ref{lem:scheme} yields
\begin{equation}\label{eq:key2}
\sup_{x\in \R^n}(u^+_N(x,t-k\eps^2) - u(x,t-k\eps^2))\leq \sup_{x\in \R^n}(u^+_N(x,t) - u(x,t))+ C\left( c_{x,2}d\eps^2+  c_{x,3}k^3\eps^3 +c_t k^2\eps^4\right)
\end{equation}
for all $k\geq d+1$ satisfying $t-k\eps^2 \geq 0$ and
\begin{equation}\label{eq:kcond}
c_{x,2}k\eps \leq C\cq \cg,
\end{equation}
where
\[c_t = \frac{C\Lip(g)}{(1-t)^{3/2}}, \ \ c_{x,2} = \frac{C\Lip(g)}{\sqrt{(1-t)\Cr}}, \ \ \text{ and }\ \ c_{x,3} = \frac{C\Lip(g)}{(1-t)\Cr}.\]
Note we can omit the error term  $c_{x,2}^2\cg^{-1}k^2\eps^3$ due to \eqref{eq:translation} and Remark \ref{rem:Hm}. Upon restricting $t \leq 1-M\eps^2$, we have $c_{x,2} \leq \frac{C\Lip(g)}{\sqrt{M\eps^2\Cr}}$,  and so then \eqref{eq:kcond} becomes
\begin{equation}\label{eq:kcondequiv}
k \leq \frac{C\cg\cq\sqrt{M \Cr}}{\Lip(g)}.
\end{equation}
We assume from now on that $k,M\in \N$ satisfy $M\geq k\geq d+1$ and \eqref{eq:kcondequiv} holds.

Let $0 \leq j \leq k-1$ and $\ell\geq 1$. Then by applying \eqref{eq:ugbase2} and then \eqref{eq:key2} $\ell$ times, we obtain
\begin{align}\label{eq:partial}
&\sup_{x\in \R^n}(u^+_N(x,1-(M + j+\ell k)\eps^2) - u(x,1-(M+j+\ell k)\eps^2))\\
&\hspace{0.15in}\leq C\Lip(g)M\eps + C\Lip(g)\sum_{i=0}^{\ell-1} \Bigg[\frac{d\eps^2}{\sqrt{(M+j+ik)\eps^2\Cr}}+ \frac{k^3\eps^3}{(M+j+ik)\eps^2\Cr}\notag\\
&\hspace{4in} + \frac{k^2\eps^4}{( (M+j+ik)\eps^2)^{3/2}} \Bigg]\notag \\
&\hspace{0.15in}= C\Lip(g)M\eps + C\Lip(g)\sum_{i=0}^{\ell-1} \left[\frac{d\eps}{\sqrt{(M+j+ik)\Cr}}+ \frac{k^3\eps}{(M+j+ik)\Cr} + \frac{k^2\eps}{(M+j+ik)^{3/2}} \right].\notag
\end{align}
Since $\ell\geq 1$ and $M\geq k$ we have
\begin{align*}
\sum_{i=0}^{\ell-1}\frac{1}{\sqrt{M+j+ik}}&\leq \frac{1}{\sqrt{k}} + \frac{1}{\sqrt{k}}\sum_{i=1}^{\ell-1}\frac{1}{\sqrt{i}}\leq \frac{1}{\sqrt{k}}+\frac{1}{\sqrt{k}}\int_{0}^{\ell-1}\frac{1}{\sqrt{x}}\, dx\leq C\sqrt{\frac{\ell}{k}},
\end{align*}
\begin{align*}
\sum_{i=0}^{\ell-1}\frac{1}{M+j+ik}&\leq\frac{1}{k}+\frac{1}{k}\sum_{i=1}^{\ell-1}\frac{1}{i}\leq\frac{C(\log(\ell+1))}{k},
\end{align*}
and
\begin{align*}
\sum_{i=0}^{\ell-1}\frac{1}{(M+j+ik)^{3/2}}&\leq\frac{1}{k^{3/2}}+\frac{1}{k^{3/2}}\sum_{i=1}^{\ell-1}\frac{1}{i^{3/2}} \leq \frac{C}{k^{3/2}}.
\end{align*}
Inserting these bounds into \eqref{eq:partial} we have
\begin{align*}
&\sup_{x\in \R^n}(u^+_N(x,1-(M+j+\ell k)\eps^2) - u(x,1-(M+j+\ell k)\eps^2))\\
&\hspace{2in}\leq C\Lip(g)M\eps + C\Lip(g)\left(d\eps\sqrt{\frac{\ell}{k\Cr}}+\frac{k^2\eps\log(\ell+1)}{\Cr} + \sqrt{k}\eps\right).
\end{align*}
Now, every $t\in [0,1-2M\eps^2]$ satisfies $\lceil t\rceil = 1 - (M+j+\ell k)\eps^2$ for some $0 \leq j \leq k-1$ and $\ell\geq 1$. Hence, we can use $\ell k\eps^2 \leq 1$ above to obtain
\[\sup_{x\in \R^n}(u^+_N(x,t) - u(x,t))\leq C\Lip(g)M\eps + C\Lip(g)\left(\frac{d}{k\sqrt{\Cr}}+\frac{k^2\eps}{\Cr}\log\left(1+ \frac{1}{k\eps^2}\right)  + \sqrt{k}\eps \right)\]
for all $t \in [0,1-2M\eps^2]$. The estimate above also holds for $t \in [1-2M\eps^2,1]$, due to \eqref{eq:ugbase2}.  Optimizing over $k$ we have $k=\lceil d^{1/3}\Cr^{1/6}\eps^{-1/3}\rceil$ which yields
\begin{align*}
&\sup_{x\in \R^n}(u^+_N(x,t) - u(x,t))\leq C\Lip(g)M + C\Lip(g)\left( 1 + \log\left(1+d^{-1/3}\Cr^{-1/6}\eps^{-5/3}\right)\right) \Cr^{-2/3}d^{2/3}\eps^{1/3}.
\end{align*}
To ensure that \eqref{eq:kcondequiv} holds and $M\geq k$, we choose 
\[M = \frac{k^2\Lip(g)^2}{C^2\cg^2\cq^2\Cr} = \frac{d^{2/3}\Cr^{-2/3}\eps^{-2/3}\Lip(g)^2}{C^2\cg^2\cq^2},\]
where $C$ is given in \eqref{eq:kcondequiv}, and we require that 
\begin{equation}\label{eq:klower}
1 \leq \frac{M}{k} = \frac{k\Lip(g)^2}{C^2\cg^2\cq^2\Cr} \iff k \geq \frac{C^2\cg^2\cq^2\Cr}{\Lip(g)^2}.
\end{equation}
This yields
\begin{align*}
&\sup_{x\in \R^n}(u^+_N(x,t) - u(x,t))\leq C\Lip(g)\left( 1 + \frac{\Lip(g)^2}{\cg^2\cq^2}+\log\left(1+d^{-1/3}\Cr^{-1/6}\eps^{-5/3}\right) \right) \Cr^{-2/3}d^{2/3}\eps^{1/3}.
\end{align*}
Since $\cq \leq 2n$, $\Cr\leq 1$ and $\cg\leq \frac{1}{n}$, the condition $k \geq \frac{4C^2}{\Lip(g)^2}$ implies \eqref{eq:klower}. In fact, by \eqref{eq:translation} we have $\Lip(g)\geq 1/\sqrt{n}$ and so $k \geq 4nC^2$ implies \eqref{eq:klower}. Since $k=\lceil d^{1/3}\Cr^{1/6}\eps^{-1/3}\rceil$ this amounts to $\eps \leq c\,d\Cr^{1/2} $ for $c>0$. Similarly, the condition $k \geq d+1$ amounts to $\eps \leq d\Cr^{1/2}/(d+1)^3$, and so we require 
\[\eps \leq d\Cr^{1/2}\min\left\{c,\frac{1}{(d+1)^3}  \right\}.\]
Since we can take $c<1$, the condition above is implied by the restriction  $\eps \leq \frac{c\,d\Cr^{1/2}}{(d+1)^3}$, This is equivalent to \eqref{eq:Ncond2} since $\eps=N^{-1/2}$.

A similar argument yields
\begin{align*}
&\inf_{x\in \R^n}(u^-_N(x,t) - u(x,t))\geq -C\Lip(g)\left( 1 + \frac{\Lip(g)^2}{\cg^2\cq^2}+\log\left(1+d^{-1/3}\Cr^{-1/6}\eps^{-5/3}\right) \right) \Cr^{-2/3}d^{2/3}\eps^{1/3}
\end{align*}
for all $t\in[0,1]$, under the same condition on $\eps$. This completes the proof.
\end{proof}

\section{Conclusion}
\label{sec:con}

This paper addresses the history-dependent prediction problem in the general case of any number of experts $n \geq 2$ and any $d \geq 1$ days of history. We prove that the rescaled value function \eqref{eq:uN} converges to the unique solution of a degenerate elliptic PDE \eqref{eq:PDE}, with convergence rates of $O(N^{-1/6})$, up to logarithmic factors. Using this result, we derived strategies for the investor that are provably asymptotically optimal. Future work will look at numerical methods for solving the PDE \eqref{eq:PDE} in order to use these results in practice, and whether we can improve the convergence rates to $O(N^{-1/2})$ to match the results from prior work \cite{drenska2019PDE} for $n=2$ and $d\leq 4$.

\end{document}